\documentclass[letterpaper,10pt,journal,onecolumn]{ieeetran}
\usepackage{graphicx}
\usepackage{gensymb}
\usepackage{amsmath, amssymb, amsthm}
\usepackage{pdfpages}
\usepackage{algorithm}
\usepackage[noend]{algpseudocode}
\usepackage{verbatim}
\usepackage{booktabs}
\usepackage{textcomp}
\newtheoremstyle{assumptionstyle}{0in}{0in}{\normalfont}{0.5em}{\itshape}{:}{.5em}{}
\newtheoremstyle{propertystyle}{0in}{0in}{\normalfont}{}{\bf}{:}{.5em}{}
\theoremstyle{plain}
\newtheorem{theorem}{Theorem}
\theoremstyle{propertystyle}
\newtheorem{assumption}{Assumption}
\theoremstyle{propertystyle}
\newtheorem{property}{Property}
\usepackage{setspace}
\usepackage{moresize}
\usepackage{lipsum}

\newcommand\blfootnote[1]{%
  \begingroup
  \renewcommand\thefootnote{}\footnote{#1}%
  \addtocounter{footnote}{-1}%
  \endgroup
}

\begin{document}
\setlength{\abovedisplayskip}{3pt}
\setlength{\belowdisplayskip}{3pt}

\title{Reactive Trajectory Generation for Multiple Vehicles in Unknown Environments with Wind Disturbances}

\author{Kenan Cole,
\thanks{Kenan Cole is a graduate student in the Mechanical and Aerospace Engineering Department,     The George Washington University, Washington, DC 20052, USA}
        ~Adam M. Wickenheiser%
\thanks{Adam M. Wickenheiser is an assistant professor in the Mechanical and Aerospace Engineering Department, The George Washington University,
        Washington, DC 20052, USA}}
        
\maketitle

\begin{abstract}
Unmanned aerial vehicle (UAV) use continues to increase, including operating beyond line of sight in unknown environments where the vehicle must autonomously generate a trajectory to safely navigate. In this article, we develop a trajectory generation algorithm for vehicles with second-order dynamics in unknown environments with bounded wind disturbances where the vehicle only relies on its on-board distance sensors and communication with other vehicles to navigate. The proposed algorithm generates smooth trajectories and can be used with high-level planners and low-level motion controllers. The algorithm computes a maximum safe cruise velocity for the vehicle in the environment and guarantees that the trajectory does not violate the vehicle's thrust limitation, sensor constraints, or user-defined clearance radius around other vehicles and obstacles. Additionally, the trajectories are guaranteed to reach a stationary goal position in finite time given a finite number of bounded obstacles. Simulation results demonstrate the algorithm properties through two scenarios: (1) a quadrotor navigating through a moving obstacle field to a goal position, and (2) multiple quadrotors navigating into a building to different goal positions. 
\end{abstract}

\section{Introduction}
Unmanned aerial vehicles (UAVs) continue to become more prolific, with a new focus on enabling autonomous navigation. The push for beyond-line-of-sight (BLOS) operation is becoming more of a reality with improved sensors such as miniature radars weighing as little as 120g with ranges on the order of hundreds of meters \cite{Echodyne},\cite{Newmeyer2016}. Additionally, laser range finders weighing as little as 120g provide 360\degree~coverage and ranges up to 40m \cite{SweepV1}. Associated with using this technology are the challenges of autonomous sense and avoid, how to operate in unknown and potentially harsh environments, and how to compensate for hardware constraints such as maneuverability and sensor limitations. These constraints are particularly important for vehicles with second-order dynamics where the vehicle cannot turn instantaneously, so the trajectory generation algorithm must compensate. Collision-free trajectory generation to a goal position for each vehicle under hardware limitations is the focus of this article.  

There are several approaches to trajectory generation in the presence of obstacles and/or other vehicles. Hoy, et.~al.~\cite{Hoy2014} provide a good summary article of various approaches and desirable algorithm features. The most popular approaches include global planners, local and reactive planners, and formation controllers. In the trajectory generation literature, global optimization techniques are prevalent \cite{Turpin2014,Bhattacharya2015,VanLoock2014} because for a known environment, they can ensure convergence to the goal position. Global optimization is not possible for our application where the environment is unknown and dynamic.

Local planners are similar to global planners but examine a shorter time window to reduce the computational expense. They can also address obstacles that may not be known \textit{a priori}\normalfont. For example, Alonso-Mora et al.~\cite{AlonsoMora2015} take the trajectory from a global planner and locally modify it to address any additional constraints based on other vehicle motion. Shiller et al.~\cite{Shiller2013} take a similar approach by optimizing the trajectory around immediate obstacles. One of the main drawbacks to local planners is the lack of an overall safety or convergence guarantee since the optimization is occurring for short time windows for only the closest obstacles. 

Reactive controllers are a type of local planner that generate the trajectory directly as the environment is sensed. These approaches utilize distance sensors to determine course changes \cite{Chunyu2010, Tang2013, Matveev2015}. While these solutions are generally not optimal, they are typically computationally faster than the optimized solutions and do not require convergence of an optimization algorithm to generate a viable solution. Their drawback however is that they do not address the smoothness of the trajectory. This can be problematic if the desired navigation requires more thrust than the vehicle can produce, and/or if the higher derivatives of the trajectory are not bounded, which may violate vehicle controller requirements. 

Formation controllers typically govern the motion of multiple vehicles using a reduced set of parameters or states and also provide solutions for collision avoidance. Examples of collision avoidance methods include potential fields \cite{Chang2015}, decentralized cooperation through sharing possible trajectory sets \cite{Bekris2012}, and navigation of the formation as a rigid body \cite{AlonsoMora2015,Zhang2010, Sarkar2013}. In the scenario we consider, the behavior is more similar to swarms, which may change composition and formation and are defined by only a few parameters. Groups such as \cite{Belta2004} consider swarm behavior in obstacle-free environments, whereas \cite{Hedjar2014} relies on a distributed optimization between the vehicles to avoid obstacles and maintain the formation. Our algorithm assures the safety of vehicles in the formation and also smoothly and safely navigates re-tasked vehicles out of the formation. In addition, our algorithm can be applied to clusters of formations of vehicles with their own clearance radii.

The physical limitations of the vehicle, such as maneuverability, sensing, and control input constraints, must also be considered to ensure the generated trajectory is feasible. In the literature there are various works that consider limitations such as sensor range (\cite{Chunyu2010}, \cite{Tang2013},  \cite{Choi2013}\nocite{Roelofsen2015}\nocite{Hoy2012}-\cite{Ferrera2017}), maximum velocity (\cite{Tang2013}, \cite{Matveev2015}, \cite{Bekris2012},\cite{Hoy2012}, \cite{Ferrera2017}), clearance radius (\cite{AlonsoMora2015}, \cite{Chunyu2010}, \cite{Matveev2015}, \cite{Choi2013}, \cite{Hoy2012}, \cite{Ferrera2017}), and turning rate (\cite{Chunyu2010}, \cite{Choi2013}, \cite{Hoy2012}, \cite{Ferrera2017}). Setting bounds on only a subset of these parameters may be reasonable for certain environments; however, all parameters are important in potentially harsh and unknown environments to ensure the trajectory is not too aggressive. Of the works reviewed, only a few consider all of these constraints simultaneously, but none consider environmental disturbances as input to the trajectory generation. Examination of disturbances is much more prevalent in vehicle controller literature to show ultimate bounded or asymptotic stability \cite{Cabecinhas2013}\nocite{Waslander2009}\nocite{Sydney2013}-\cite{FischerNL2014}. To achieve these stability guarantees, the controllers require the desired trajectory higher derivatives to exist and be bounded. In order to meet these criteria, the control authority to overcome the disturbance must also be considered when generating the trajectory.




To address each of these areas, we build upon \cite{Cole2017}, which describes trajectory generation for groups of quadrotors in unknown environments that bounds the maximum cruise velocity and respects thrust limitations and sensor constraints. In this article, we establish guarantees with respect to the obstacle spacing, which was not explicitly considered in the prior work. Additionally, we look at the maneuverability of vehicles separate from the obstacles to enable more aggressive maneuvering. Lastly, we account for the goal position more explicitly when determining course and velocity changes to reduce the trajectory length/time.

We organize the rest of the paper by first defining the problem, algorithm properties, and operating assumptions in Sec.~\ref{SecProbDef}. The trajectory generation is defined in Sec.~\ref{SecTrajGen}, which describes how to smoothly adjust vehicle course and/or velocity to safely clear obstacles and other vehicles. Section \ref{SecSafetyGuarantee} provides the analysis for bounding the trajectory acceleration to respect thrust limitations and bounding the maximum cruise velocity to safely navigate the environment. The vehicle dynamics and controller for the simulation case study are given in Sec.~\ref{SecVehController}. Two simulation case studies in Sec.~\ref{SecSim} demonstrate the trajectory generation algorithm's features. Finally Sec.~\ref{SecConclusion} provides concluding remarks. 

\section{Problem Definition}
\label{SecProbDef}
We define an algorithm that generates a trajectory for each vehicle that satisfies Properties 1 and 2 for an environment similar to Fig.~\ref{FigDetObsInPath}. These properties are rigorously achieved, as shown in Sec.~\ref{SecSafetyGuarantee}, under the following assumptions, some of which may be relaxed as discussed in Sec.~\ref{SecConclusion}.

\subsection{Algorithm Properties}
\label{SubSecAlgProps}
\begin{property} Generation of a piecewise-smooth (with isolated bounded discontinuities) desired trajectory $\mathbf{p}_d \in \mathbb{R}^3$ where the derivatives $\mathbf{p}_d^{(i)} \in \mathbb{R}^3,~\forall i=0,1,\ldots,n$ exist, are bounded, and respect the vehicle's maximum thrust, $f_{max}$, for a translational wind velocity of unknown direction and bounded magnitude, $||\mathbf{v}_{air}|| \leq v_{air,max}$.
\end{property}
\begin{property} Clearance of all obstacles and other vehicles by a user-defined clearance radius, $r_c$, which takes into account the vehicle's size as well as measurement, estimation, and tracking errors. 
\end{property}

\subsection{Algorithm Assumptions}
\begin{assumption}
\label{AssumptionPlanar}
	Vehicle desired trajectories and obstacle motions are planar, but vehicle dynamics are not restricted to be planar.
	\end{assumption}
	\begin{assumption}
	\label{AssumpVehicles}
	 Vehicles are finite in number and heterogeneous in physical parameters (mass, max thrust, etc.) and importance (e.g. higher valued asset).  
	\end{assumption}
	\begin{assumption}
	 \label{AssumpVehCommsInfo} Vehicles share current position and course information when in range via wireless communication.
	\end{assumption} 
	\begin{assumption}
	 \label{AssumpSensorComms} Vehicles sensor and communication sample periods and ranges are equal and given by $\Delta T_s$ and $r_s > r_c$, respectively. Within these limitations, the sensor and inter-vehicle communications provide perfect distance and velocity information. 
	\end{assumption}
	\begin{assumption}
	  The clearance radius $r_c$ ensures there are no aerodynamic interactions between one vehicle and another or with obstacles. 
	\end{assumption}
	\begin{assumption}
	 Wind disturbances are bounded, time-varying, and planar. Updraft effects near obstacles are assumed to be limited to a distance less than $r_c$.
	\end{assumption}
	\begin{assumption}
	\label{AssumpLessCapable}
	There are a finite number of obstacles and each obstacle is finite size, moves with constant velocity (less than minimum vehicle cruise velocity) and constant course. Minimum obstacle separation does not prevent the vehicles from moving between them.
	\end{assumption}
	\begin{assumption}
	\label{AssumpGoalPos}
	 Goal positions are not too close to obstacles or each other to violate vehicle clearance radii and are not infinitely far from the coordinate origin.
	\end{assumption}

\begin{figure}
	\centering
		\includegraphics[width=2.95in]{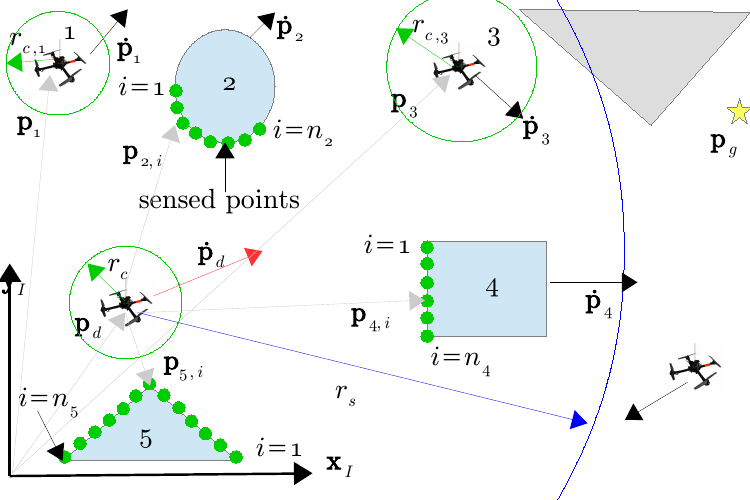}
		\caption{\label{FigDetObsInPath} Representative environment where a vehicle navigates around other vehicles and obstacles to reach a goal position (yellow star). Potential collisions are prioritized based on distance and course angle.}
\end{figure}

\section{Trajectory Generation}
\label{SecTrajGen}
The trajectory generation algorithm takes each vehicle from its starting position and velocity and guides it on a collision-free trajectory to the goal position. To achieve this, the vehicle first determines which vehicles in $r_s$ it is responsible for maneuvering around. Next, if the vehicle has thrust availability for maneuvering, it compiles all sensor/communication inputs to identify the most imminent obstacle/vehicle safety threat. The vehicle then computes a circumnavigation direction to traverse the obstacle/vehicle, a course change angle, and a velocity change to maintain the desired clearance radius, $r_c$. Finally, the vehicle uses sigmoid functions to smoothly transition to the desired course and velocity. These steps are discussed in detail in Secs.~\ref{SubSecVehManeuver} to \ref{SecSigmoid}.

\subsection{Ranking Vehicles' Maneuverability}
\label{SubSecVehManeuver}
To determine which vehicles maneuver and which vehicles stay on course, vehicles exchange their maximum cruise velocity, $v_c$, current velocity, $\mathbf{\dot{p}}_d$, clearance radius, $r_c$, and a pre-assigned $ID$ value when they come within communication range of each other. To satisfy Assumption \ref{AssumpLessCapable}, vehicles with larger $v_c$ must maneuver around vehicles with smaller $v_c$. If the vehicles have equal $v_c$ values, then the vehicles with lower $ID$ values maneuver around vehicles with higher $ID$ values, forming the set $\mathcal{I}_{mnvr} \subseteq \mathcal{I}_{nr}$, where $\mathcal{I}_{nr}$ is the set of all vehicles within $r_s$ of the vehicle's current position. Hovering or loitering vehicles are considered to have $v_c = 0$.

\subsection{Compiling Sensor Inputs}
\label{SubSecObsVehIdentify}
The vehicle uses distance and angle measurements to obstacles and other vehicles to determine the most imminent collisions, if any. We assume that the sensing is isotropic (i.e. has the same range and rate in all directions). The sensor output is a data array of relative positions of sensed points on obstacles. By finding discontinuities in range and angle, the sensor scan information is used to distinguish different obstacles, each of which is given a unique local identifier, $id$. The $id$ values of all obstacles within the sensor scan comprise the set $\mathcal{I}_{obs} = \left\{id_1,...,id_m \right\}$, where $m$ is the number of distinct obstacles within range. The inertial positions of the sensed points are given by $\mathbf{p}_{id,i}$, where $i = 1,...,n_{id}$, and $n_{id}$ is the number of sensed points for that particular obstacle.

The inter-vehicle communication provides inertial positions, $\mathbf{p}_{ID}$, in addition to the data described in Sec.~\ref{SubSecVehManeuver}. The data for the vehicles in $\mathcal{I}_{mnvr}$ is combined with the data for the obstacles in $\mathcal{I}_{obs}$ to form a data array of distinct vehicles and obstacles that is used to determine appropriate course and/or velocity changes for collision-free navigation in the environment.

\subsection{Critical Obstacle and Vehicle Identification}
\label{SecCriticalObsVehID} 
Now that the vehicle has compiled its sensor and communication inputs, it identifies critical and non-critical obstacles/vehicles in the environment. Critical obstacles/vehicles are within the minimum reaction distance (defined in Eq.~\ref{EqRcStar}), require immediate action from the vehicle to avoid collisions and violations of $r_c$, and if there are multiple critical obstacles/vehicles then all contribute to the course change. Non-critical obstacles are outside the minimum reaction distance and contribute to course changes when there are no critical obstacles or the critical obstacles do not prohibit the vehicle from navigating to the goal position. This section describes the process to determine the sets $\mathcal{I}_{co}$, $\mathcal{I}_{cv}$, $\mathcal{I}_{no}$, and $\mathcal{I}_{nv}$, which are the sets of critical obstacles and vehicles and non-critical obstacles and vehicles, respectively.

The vehicle first determines the closest sensed point for the $k^{th}$ obstacle/vehicle in $\mathcal{I}_{obs} \cap \mathcal{I}_{nr}$
\begin{equation}
	\label{EqPMin}
	\mathbf{p}_{k,min} = \underset{i}{\mathrm{argmin}}\left(||\mathbf{p}_{k,i} - \mathbf{p}_d||\right)
\end{equation}
\noindent where $i = 1,...,n_{k}$ are the indices of the sensed points for obstacle/vehicle $k$. 

Since the other vehicles can change course and speed whereas the obstacles have constant course and speed, the minimum reaction distance to maintain $r_c$ and avoid collisions is different for obstacles and vehicles. We define the minimum reaction distance to avoid obstacle/vehicle $k$ as
\begin{align}
	\label{EqRcStar}
	r_{c,k}^* =  \left\{\begin{array}{ll}
	r_{obs}, & k \in \mathcal{I}_{obs} \\
	\mathrm{max}(r_{c,k},r_c) + r_{180} + ||\mathbf{\dot{p}}_{k}|| \Delta T_{tot}, & k \in \mathcal{I}_{nr}
	\end{array}
	\right.
\end{align}
\noindent where $r_{obs}$ is the minimum distance between obstacles, $r_c$ is the clearance radius of the current vehicle, $r_{c,k}$ is the clearance radius of sensed vehicle $k$, $\Delta T_{tot}=\tau_{f,180} + 2\Delta T_s + \Delta T_c$, $\mathbf{\dot{p}}_k$ is the velocity vector of vehicle $k$, $r_{180} = v_c \int_0^{\tau_{f,180}} \sin \phi(t) dt$ and $\tau_{f,180} = c_3/a_{max} \pi v_c$ are the distance and time span required for the vehicle to make a 180\degree~turn ($v_c$ is defined in Theorem 2, $\phi(t),c_3,a_{max}$ are defined in Sec.~\ref{SecSigmoid}), and $\Delta T_c < \Delta T_s$ is the maximum on-board algorithm computation time. Appendix C discusses the development of the definition of $r_{c,k}^*$. 

To accurately determine which obstacle/vehicle poses the most imminent threat, we normalize the distance from the vehicle to $\mathbf{p}_{k,min}$ by the corresponding minimum reaction distance, $r_{c,k}^*$. The obstacle/vehicle that minimizes Eq.~\ref{EqPMinGlobal} is the most imminent threat:
\begin{equation}
	\mathbf{p}_{min} = \mathbf{p}_{k_{min},min}
\end{equation} 
\noindent where
\begin{equation}
	\label{EqPMinGlobal}
	k_{min}= \underset{k}{\mathrm{argmin}}\left(\frac{||\mathbf{p}_{k,min}-\mathbf{p}_d||-r_{c,k}^*}{r_{c,k}^*}\right)
\end{equation}

Since $r_c < r_{c,k}^*$ for all cases, $||\mathbf{p}_{k,min} - \mathbf{p}_d||-r_{c,k}^*$ can be negative and likely is negative while traversing an obstacle/vehicle. When it is negative, the obstacle/vehicle that the vehicle is traversing is defined as critical. The sets of critical obstacles and vehicles are respectively defined as
\begin{align}
	\mathcal{I}_{co} &= \left\{k| ~||\mathbf{p}_{k,min}-\mathbf{p}_d||<r_{c,k}^*, k \in \mathcal{I}_{obs} \right\} \\
	\mathcal{I}_{cv}& = \left\{k| ~||\mathbf{p}_{k,min}-\mathbf{p}_d||<r_{c,k}^*, k \in \mathcal{I}_{nr} \right\} 	
\end{align}

The non-critical sets of obstacles and vehicles are, $\mathcal{I}_{no} = \mathcal{I}_{obs} - \mathcal{I}_{co}$ and $\mathcal{I}_{nv} = \mathcal{I}_{nr} - \mathcal{I}_{cv}$, respectively. These sets are used for determining appropriate course changes as discussed in Sec.~\ref{SecCourseChange}.

\subsection{Course Change Definition for an Obstacle}
\label{SecCourseChangeOneObs}
To safely navigate the environment and avoid collisions, the vehicle can change course and/or velocity. In this study, we assume that the vehicle travels at its maximum safe cruise velocity and makes course changes as the default behavior to try to minimize the time required to reach the goal position. The process for determining an appropriate course change applies to both critical and non-critical obstacles. This section describes the process to determine a course change angle, $\Delta \phi_k$, and the set of all feasible course angles, $\mathbf{O}_{k}$, for obstacle $k$. There are a few differences in the process for obstacles compared to vehicles, so vehicles are discussed separately in Sec.~\ref{SecCourseChangeOneVeh}.

To start the process of determining a candidate course change, the vehicle takes the sensed points for each obstacle and determines the bounding extent points (i.e.~the left- and right-most extent points), $\mathbf{p}_{k,e1}$ and $\mathbf{p}_{k,e2}$, and their corresponding projected extent points, $\mathbf{p}_{k,e1}^*$ and $\mathbf{p}_{k,e2}^*$ which take into account $r_c$. This process is illustrated in Fig.~\ref{FigPadObsByRC} where the vehicle first computes the angle to the projected sensed points as follows: 
\begin{equation}
	\label{EqPhiEi}
	\phi_{e,i} = \phi_{e1,i} + \phi_{e2,i}
\end{equation}
\noindent where
\begin{align}	
	&\phi_{e1,i} = \mathrm{angle}(\mathbf{r}_{k,min},\mathbf{r}_{k,i}) \\
	\label{EqPhiE2PadObs}
	&\phi_{e2,i} = k_{\phi,e1,i}\left(\phi_{e1,i}\right)\sin^{-1}\left(\frac{r_c}{||\mathbf{r}_{k,i}||}\right) \\
	&\mathbf{r}_{k,i} = \mathbf{p}_{k,i}-\mathbf{p}_d \\
	&k_{\phi,e1,i} = \left\{\begin{array}{ll}
	\mathrm{sgn}(\phi_{e1,i}), & |\phi_{e1,i}| > 0 \\
	\mathrm{sgn}(\mathrm{angle}(\mathbf{r}_{k,min},\mathbf{\dot{p}}_d)), & \phi_{e1,i} = 0
	\end{array}
	\right. \\
	\label{EqPkiStar}
	&\mathbf{p}_{k,i}^* = \mathbf{p}_d + \sqrt{||\mathbf{r}_{k,i}||^2 - r_c^2}\mathbf{R}_{\phi_{e,i}}\hat{\mathbf{r}}_{k,min}
\end{align}
\noindent where $i = 1, ..., n_{k}$ is the index of sensed points for obstacle $k$. Note Eq.~\ref{EqPhiE2PadObs} only produces a real result when $||\mathbf{r}_{k,i}|| \geq r_c$; however, Sec.~\ref{SecSafetyGuarantee} guarantees this condition.

\blfootnote{1.~$\mathbf{\hat{x}}=\mathbf{x}/||\mathbf{x}||$ for any vector $\mathbf{x}$}
\blfootnote{2.~$\mathbf{R}_{\theta}$ is the rotation matrix to make a $\theta$ rotation about $\mathbf{z}_I$}
\blfootnote{3.~$\mathrm{angle}(\mathbf{p}_1,\mathbf{p}_2)$ returns the signed angle from $\mathbf{p}_1$ to $\mathbf{p}_2$.}

The bounding extent points are the points that produce the maximum and minimum $\phi_{e,i}$ as shown in Fig.~\ref{FigPadObsByRC}B and defined as
\begin{align}
	\label{EqPE1}
	&\mathbf{p}_{k,e1} = \underset{i}{\mathrm{argmax}}\left(\phi_{e,i}\right) \\
	\label{EqPE2}	
	&\mathbf{p}_{k,e2} = \underset{i}{\mathrm{argmin}}\left(\phi_{e,i}\right)
\end{align}

The final point the vehicle calculates is the projected minimum point as shown in Fig.~\ref{FigPadObsByRC}B and defined as:
\begin{equation}
	\mathbf{p}_{k,min}^* = \mathbf{p}_{k,min} + r_c \left(-\hat{\mathbf{r}}_{k,min}\right)
\end{equation}

If there is only one sensed point for an obstacle, then the extent points and projected extent points are equal as shown in Fig.~\ref{FigPadObsByRC}C and defined as
\begin{align}
	\label{EqPke1SinglePt}
	\mathbf{p}_{k,e1} =\mathbf{p}_{k,e1}^*=\mathbf{p}_{k,1} + r_c \mathbf{R}_{+\phi_{h}} \left(-\hat{\mathbf{r}}_{k,1}\right) \\
	\label{EqPke2SinglePt}
	\mathbf{p}_{k,e2} =\mathbf{p}_{k,e2}^*=\mathbf{p}_{k,1} + r_c \mathbf{R}_{-\phi_{h}} \left(-\hat{\mathbf{r}}_{k,1}\right)
\end{align}
\noindent where
\begin{equation}
	\label{EqPhiHObs}
	\phi_h = \cos^{-1} \left(\frac{r_c}{||\mathbf{r}_{k,1}||}\right)
\end{equation}
\begin{figure}
	\centering
		\includegraphics[width=2.95in]{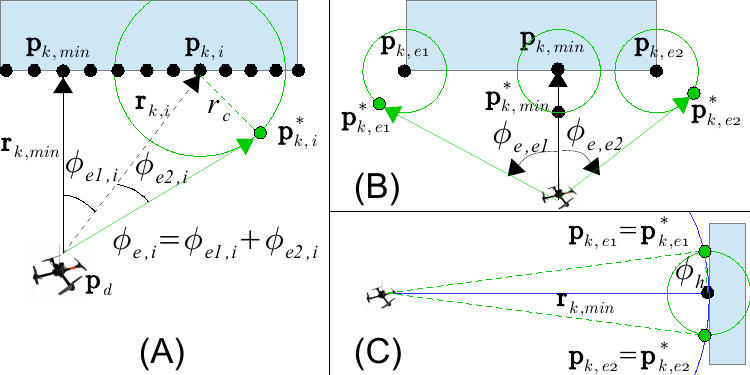}
		\caption{\label{FigPadObsByRC} (A) Determination of $\phi_{e,i}$ for each sensed point. (B) Determination of the bounding extent and projected extent points from Eqs.~\ref{EqPE1} and \ref{EqPE2}. (C) Determination of extent and projected extent points for a single sensed point.}
\end{figure}
Next, the vehicle uses the projected extent points to determine four candidate tangent directions per obstacle. The candidate tangent directions are shown in Fig.~\ref{FigFeasibleSet}A and summarized as
\begin{align}
	\label{EqPljS1}
	\mathbf{p}_{k,s1} &= \mathbf{p}_{k,e1}^* - \mathbf{p}_{k,min}^* \\
	\mathbf{p}_{k,s2} &= \mathbf{p}_{k,e1}^* - \mathbf{p}_d \\
	\mathbf{p}_{k,s3} &= \mathbf{p}_{k,e2}^* - \mathbf{p}_{k,min}^* \\
	\label{EqPljS4}
	\mathbf{p}_{k,s4} &= \mathbf{p}_{k,e2}^* - \mathbf{p}_d
\end{align}
The $\mathbf{p}_{k,s1}$ and $\mathbf{p}_{k,s3}$ tangent directions are ``conservative" because they define the slope based on the projected minimum point, thus keeping the vehicle parallel with the estimated obstacle ``face". The $\mathbf{p}_{k,s2}$ and $\mathbf{p}_{k,s4}$ tangent directions are ``aggressive" because they allow the vehicle to get closer to the obstacle by heading towards the tangent point on the $r_c$ circle. 

We use the tangent directions to determine the obstacle velocity components parallel to the tangent directions (i.e.~along the obstacle ``face"), $\mathbf{p}_{k,\parallel,i}$, and perpendicular to the tangent directions (i.e.~normal to the obstacle ``face"), $\mathbf{p}_{k,\perp,i}$ for $i=1,\ldots,4$. The vehicle must match the component of velocity in the $\mathbf{p}_{k,\perp,i}$ direction to avoid collisions, then uses any remaining velocity to traverse the obstacle in the $\mathbf{p}_{k,\parallel,i}$ direction. We define these quantities in the following paragraphs, where Eqs.~\ref{EqPpar} to \ref{Eqzt} are evaluated for all four candidate tangent directions, but for brevity, only the $\mathbf{p}_{k,s1}$ equations are presented. 

The unit vectors parallel and perpendicular to the tangent direction and the corresponding obstacle velocity components are respectively given by
\begin{align}
	\label{EqPpar}
	\mathbf{p}_{k,\parallel,s1} &= \hat{\mathbf{p}}_{k,s1} \\
	\label{EqPperp}
	\mathbf{p}_{k,\perp,s1} &= R_{\theta_{k,s1}} \hat{\mathbf{p}}_{k,s1} \\
	\label{EqVParll}
	v_{k,\parallel,s1} &= \mathbf{\dot{p}}_{k} \cdot \mathbf{p}_{k,\parallel,s1}\\
	\label{EqVperp}
	v_{k,\perp,s1} &= \mathbf{\dot{p}}_{k} \cdot \mathbf{p}_{k,\perp,s1}
\end{align}
\noindent where
\begin{equation*}
\theta_{k,s1} = \mathrm{sgn}\left(\left(\mathbf{p}_{k,s1} \times \left(\mathbf{p}_d - \mathbf{p}_{k,min}\right)\right)\cdot \mathbf{z}_I\right) \frac{\pi}{2}
\end{equation*}
To avoid collisions, the navigating vehicle matches at minimum, the obstacle velocity component in the $\mathbf{p}_{k,\perp,s1}$ direction, so that Eq.~\ref{EqMatchOrthoComponent} is satisfied. The remaining velocity magnitude, defined in Eq.~\ref{EqVrem}, is available to traverse the obstacle:
\begin{align}
	\label{EqMatchOrthoComponent}
	&\mathbf{\dot{p}}_d \cdot \mathbf{p}_{k,\perp,s1} \geq v_{k,\perp,s1} \\
	\label{EqVrem}
	&v_{k,rem,s1} = \sqrt{||\mathbf{\dot{p}}_d||^2 - v_{k,\perp,s1}^2}
\end{align}
\noindent where $v_{k,rem,s1} > 0$ since Assumption \ref{AssumpLessCapable} guarantees
\begin{equation}
	||\mathbf{\dot{p}}_d|| > ||\mathbf{\dot{p}}_{k}|| \geq v_{k,\perp,s1}
\end{equation} 
Next, we define the desired velocity vector for the vehicle that matches the perpendicular component of the obstacle velocity and applies the remaining velocity to traverse the tangent direction as follows:
\begin{equation}
	\label{EqVfDef}
	\mathbf{v}_{k,s1} = v_{k,\perp,s1} \mathbf{p}_{k,\perp,s1} + v_{k,rem,s1} \mathbf{p}_{k,\parallel,s1}
\end{equation}
The corresponding course change to reach the desired velocity vector is
\begin{equation}
	\label{EqDelPhiPerTanDir}
	\Delta \phi_{k,s1} = \mathrm{angle}(\mathbf{\dot{p}}_d,\mathbf{v}_{k,s1})
\end{equation}

The circumnavigation direction corresponding to this tangent direction is defined by
\begin{equation}
	\label{EqCircumDir}
	\mathbf{z}_{k,s1} = -\mathrm{sgn}(z_{k,s1}) \mathbf{z}_I
\end{equation}
\noindent where
\begin{equation}
	\label{Eqzt}
	z_{k,s1} = \left(\left(\mathbf{p}_{k,min} - \mathbf{p}_d \right) \times \mathbf{p}_{k,s1}\right) \cdot \mathbf{z}_I
\end{equation}
The circumnavigation direction is in the $+\mathbf{z}_I$ direction if the vehicle traverses counterclockwise around the obstacle and is in the $-\mathbf{z}_I$ direction otherwise. The circumnavigation direction in Eq.~\ref{EqCircumDir} and the candidate course change in Eq.~\ref{EqDelPhiPerTanDir} define the minimum absolute course angle, but this is not the only feasible course angle. The feasible course angles for tangent direction $\mathbf{p}_{k,s1}$ include any angle between $\Delta \phi_{k,s1}$ and the angle to the aggressive tangent direction for the opposite extent point, $\Delta \phi_{k,s4}$, as shown in Fig.~\ref{FigFeasibleSet}. We define this generically as
\begin{align}
	\label{EqSetOrt}
	\mathbf{O}_{k,si}' =  \left\{\begin{array}{ll}
	\left\{\Delta \phi| \Delta \phi_{k,sj}' \leq \Delta \phi < \Delta \phi_{k,si}\right\},& \mathbf{z}_{k,si} = \mathbf{z}_I \\
	\left\{\Delta \phi|\Delta \phi_{k,si} < \Delta \phi \leq \Delta \phi_{k,sj}'\right\}, & \mathbf{z}_{k,si} = -\mathbf{z}_I
	\end{array}
	\right. 
\end{align}
\noindent where
\begin{align*}
	\label{EqPhiKSprime}
	\Delta \phi_{k,sj}' =  \left\{\begin{array}{ll}
 \Delta \phi_{k,s4}, & z_{k,si}\Delta \phi_{k,s4}>0,i=1,2 \\
	\mathrm{sgn}\left(z_{k,si}\right) (2\pi - |\Delta \phi_{k,s4}|), & z_{k,si}\Delta \phi_{k,s4}<0,i=1,2 \\
	\Delta \phi_{k,s2}, & z_{k,si}\Delta \phi_{k,s2}>0,i=3,4 \\
	\mathrm{sgn}\left(z_{k,si}\right) (2\pi - |\Delta \phi_{k,s2}|), &z_{k,si}\Delta \phi_{k,s2}<0, i=3,4
	\end{array}
	\right.
\end{align*}
\noindent The set defined in Eq.~\ref{EqSetOrt} is further refined in Eq.~\ref{EqSetOrtRefined} once a circumnavigation direction is chosen.

\begin{figure}
	\begin{center}
		\includegraphics[width=2.8in]{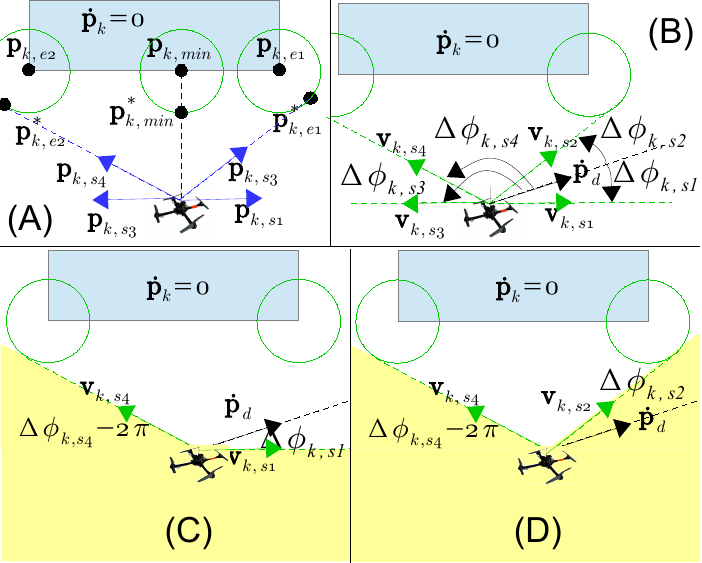}
		\caption{\label{FigFeasibleSet}(A) Tangent direction definitions for a vehicle traversing a stationary obstacle. (B) Example of four candidate course changes (for each tangent direction) for a stationary obstacle. The analysis applies to stationary and moving obstacles but is easier to visualize for a stationary obstacle. (C) The feasible set $\mathbf{O}_{k,s1} = \left\{\Delta \phi| \Delta \phi_{k,s4}-2\pi \leq \Delta \phi \leq \Delta \phi_{k,s1}\right\}$ is shaded in yellow for the candidate course change $\Delta \phi_{k,s1}$. (D) The feasible set $\mathbf{O}_{k,s2} = \left\{\Delta \phi| \Delta \phi_{k,s4}-2\pi \leq \Delta \phi \leq \Delta \phi_{k,s2}\right\}$ is shaded in yellow for the candidate course change $\Delta \phi_{k,s2}$. The sets are similar for candidate course changes $\Delta \phi_{k,s3}$ and $\Delta \phi_{k,s4}$.}
	\end{center}
\end{figure}

At this point, the vehicle has four candidate final velocity vectors defined in Eq.~\ref{EqVfDef} and must choose among these four. The vehicle can traverse the obstacle towards $\mathbf{p}_{k,e1}$ or $\mathbf{p}_{k,e2}$, where the $\mathbf{v}_{k,s1}$ and $\mathbf{v}_{k,s2}$ candidates are associated with $\mathbf{p}_{k,e1}$, and the $\mathbf{v}_{k,s3}$ and $\mathbf{v}_{k,s4}$ candidates are associated with $\mathbf{p}_{k,e2}$. 

The vehicle's objective is to reach the goal position where the course change to the goal position is given by
\begin{equation}
	\label{EqGoalPosDelPhi}
			\Delta \phi_g = \mathrm{angle}\left(\mathbf{\dot{p}}_d,\mathbf{p}_g - \mathbf{p}_d\right) 
\end{equation}
To determine if changing course to the goal position is feasible, the vehicle first considers if the obstacle's circumnavigation direction has been established. If it has, the vehicle must choose the extent point consistent with the established direction to keep the vehicle circling the obstacle in the same direction. The circumnavigation direction is established if the obstacle was previously the active obstacle, where the active obstacle is the closest obstacle that requires the vehicle to make a heading or velocity change. This is formally defined in Sec.~\ref{SecCourseChange}.

If the circumnavigation direction has not been established and $\Delta \phi_g \in \mathbf{O}_{k,si}'$ for $i=1,...,4$, then the vehicle has sufficiently traversed obstacle $k$ such that the goal position is feasible for one of the tangent directions. The extent point associated with this tangent direction is chosen and $\Delta \phi_k = \Delta \phi_g$. 

If the circumnavigation direction has not been chosen and $\Delta \phi_g \notin \mathbf{O}_{k,si}'$ for $i=1,...,4$, then the vehicle uses Eqs.~\ref{EqTtv} and \ref{EqTtv2}, based on current sensor information, to estimate how long it would take to traverse the obstacle towards the extent points: 
\begin{align}
	\label{EqTtv}
	t_{k,e1} &= \frac{||\mathbf{p}_{k,e1}-\mathbf{p}_{k,min}||}{v_{k,rem,s1} - v_{k,\parallel,s1}} \\
	\label{EqTtv2}
	t_{k,e2} &= \frac{||\mathbf{p}_{k,e2}-\mathbf{p}_{k,min}||}{v_{k,rem,s3} - v_{k,\parallel,s3}}
\end{align}
\noindent where $t_{k,e1}$ corresponds with tangent directions $\mathbf{p}_{k,s1}$ and $\mathbf{p}_{k,s2}$, and $t_{k,e2}$ corresponds with tangent directions $\mathbf{p}_{k,s3}$ and $\mathbf{p}_{k,s4}$. The vehicle chooses the quicker direction even if it turns out not to be the quickest path after more of the obstacle is sensed.

The process for choosing the extent point for obstacles is summarized by Eq.~\ref{EqPeF}. The conditions in Eq.~\ref{EqPeF} are evaluated in sequence until one of the conditions is met: 
\begin{equation}
	\label{EqPeF}
	\mathbf{p}_{k,E} = \left\{\begin{array}{ll}
	\mathbf{p}_{k,e1}, & \mathbf{z}_{k,s1} = \mathbf{z}_{k,t}, \mathbf{z}_{k,t} \neq 0 \\
	\mathbf{p}_{k,e2}, & \mathbf{z}_{k,s3} = \mathbf{z}_{k,t}, \mathbf{z}_{k,t} \neq 0 \\
	\mathbf{p}_{k,e1}, & \Delta \phi_g \in \mathbf{O}_{k,si}',~i=1,2 \\
	\mathbf{p}_{k,e2}, & \Delta \phi_g \in \mathbf{O}_{k,si}',~i=3,4 \\
	\mathbf{p}_{k,e1}, & t_{k,e1} = t_{k,e2}, \frac{|\Delta \phi_{s1} -\Delta \phi_g|}{|\Delta \phi_{s3} - \Delta \phi_g|}\leq 1 \\
	\mathbf{p}_{k,e2}, & t_{k,e1} = t_{k,e2}, \frac{|\Delta \phi_{s1} -\Delta \phi_g|}{|\Delta \phi_{s3} - \Delta \phi_g|} > 1 \\
	\mathbf{p}_{k,e1}, & t_{k,e1} <  t_{k,e2}  \\
	\mathbf{p}_{k,e2}, & \mathrm{otherwise}
	\end{array}
	\right.
\end{equation}
\noindent where $\mathbf{z}_{k,t}$ is the established circumnavigation direction (Eq.~\ref{EqCircumDir}) from when obstacle $k$ was the active obstacle. If the circumnavigation direction has not been established, then $\mathbf{z}_{k,t}= 0$.
Now that $\mathbf{p}_{k,E}$ and the circumnavigation direction have been established, two of the four candidate final velocity vectors have been eliminated. Next, the conservative or aggressive tangent direction must be selected. Since the vehicle cannot change course instantaneously, if the aggressive tangent direction solution is in the opposite direction as the circumnavigation direction (i.e.~the vehicle maneuvering takes the vehicle closer to the obstacle before achieving the desired course change), then the aggressive tangent direction is not suitable because it will violate $r_c$. If instead the tangent direction solution is in the same direction as the circumnavigation direction, then the aggressive tangent direction is suitable. Equation \ref{EqAggresConserv} summarizes this:

\begin{align}
	\label{EqAggresConserv}
	S = \left\{\begin{array}{ll}
	s2, & \mathrm{sgn}(\mathrm{angle}(\mathbf{\dot{p}}_d,\mathbf{v}_{k,s2})) = \mathrm{sgn}(z_{k,s1}), \mathbf{p}_{k,E} = \mathbf{p}_{k,e1} \\
	s1, & \mathrm{sgn}(\mathrm{angle}(\mathbf{\dot{p}}_d,\mathbf{v}_{k,s2})) \neq \mathrm{sgn}(z_{k,s1}), \mathbf{p}_{k,E} = \mathbf{p}_{k,e1} \\
	s4, & \mathrm{sgn}(\mathrm{angle}(\mathbf{\dot{p}}_d,\mathbf{v}_{k,s4})) = \mathrm{sgn}(z_{k,s3}), \mathbf{p}_{k,E} = \mathbf{p}_{k,e2} \\
	s3, & \mathrm{sgn}(\mathrm{angle}(\mathbf{\dot{p}}_d,\mathbf{v}_{k,s4})) \neq \mathrm{sgn}(z_{k,s3}), \mathbf{p}_{k,E} = \mathbf{p}_{k,e2}
	\end{array}
	\right.
\end{align}
Now that a tangent direction has been chosen, we update $\mathbf{O}_{k,S}'$ to $\mathbf{O}_{k}$ to take into account further restrictions if the vehicle is traversing a non-convex obstacle. In this case, if the full obstacle is not within the sensor range, $\mathbf{O}_{k,S}'$ may not restrict the vehicle maneuvering enough, and the vehicle may incorrectly conclude that a course change to the goal position is feasible. 

Instead, the vehicle stores the most constraining extent point, $\mathbf{p}_{k,ep}$, which is the more restrictive of either the previous most constraining extent point, $\mathbf{p}_{k,ep}^{-}$, or the non-chosen extent point from most recent sensor information as follows:
\begin{equation}
	\label{EqPkep}
	\mathbf{p}_{k,ep} = \left\{\begin{array}{ll}
	\mathbf{p}_{k,ej}, & \mathbf{z}_{k,t} = 0 \\
	\mathbf{p}_{k,ej}, & \mathbf{z}_{k,t} = \mathbf{z}_I, \Delta \phi_{k,S'} > \Delta \phi_{k,ep} \\
	\mathbf{p}_{k,ej}, & \mathbf{z}_{k,t} = -\mathbf{z}_I, \Delta \phi_{k,S'} < \Delta \phi_{k,ep} \\
	\mathbf{p}_{k,ep}^{-}, & \mathbf{z}_{k,t} = \mathbf{z}_I, \Delta \phi_{k,ep} \geq \Delta \phi_{k,S'} \\
	\mathbf{p}_{k,ep}^{-}, & \mathbf{z}_{k,t} = -\mathbf{z}_I, \Delta \phi_{k,ep} \leq \Delta \phi_{k,S'} 
	\end{array}
	\right.
\end{equation}
\noindent where
\begin{align*}
	&\Delta \phi_{k,ep} = \mathrm{angle}(\mathbf{\dot{p}}_d,\mathbf{p}_{k,ep}^{-*} - \mathbf{\dot{p}}_d) \\
	&j = \left\{\begin{array}{ll}
	1, & S = s3,s4 \\
	2, & S = s1, s2
	\end{array}
	\right.\\
	&S' = \left\{\begin{array}{ll}
	s4, & S = s1,s2 \\
	s2, & S = s3,s4
	\end{array}
	\right.
\end{align*}
\noindent and $\mathbf{p}_{k,ep}^{-*}$ is calculated from Eq.~\ref{EqPkiStar}.

The heading change associated with the most constraining extent point is compared to the feasible course change angles in $\mathbf{O}_{k,S}'$ according to
\begin{align}
	\label{EqSetOrtRefined}
	&\mathbf{O}_{k} = \left\{\begin{array}{l}
	\left\{\Delta \phi| \mathrm{max}(\Delta \phi_{k,ep},\mathrm{min}(\mathbf{O}_{k,S}')) \leq \Delta \phi < \Delta \phi_{k,S}\right\},\mathbf{z}_{k,S} = \mathbf{z}_I \\
	\left\{\Delta \phi|\Delta \phi_{k,S} < \Delta \phi \leq  \mathrm{min}(\Delta \phi_{k,ep},\mathrm{max}(\mathbf{O}_{k,S}')) \right\},\mathbf{z}_{k,S} = -\mathbf{z}_I 
	\end{array}
	\right.
\end{align} 
The final step is to define the course change angle for obstacle $k$ as 
\begin{equation}
	\label{EqDelPhiPerObs}
	\Delta \phi_k = \left\{\begin{array}{ll}
	\Delta \phi_g, & \Delta \phi_g \in \mathbf{O}_{k} \\
	\Delta \phi_{k,S}, & \mathrm{otherwise}
	\end{array}
	\right.
\end{equation}
\noindent The course change angle and set of all feasible course angles for obstacle $k$ are used to compare to other obstacles and vehicles to determine a final course change in Sec.~\ref{SecCourseChange}. Figure \ref{FigUpdatedProcess} shows this process for a circular obstacle.

\begin{figure}
	\centering
		\includegraphics[width=2.5in]{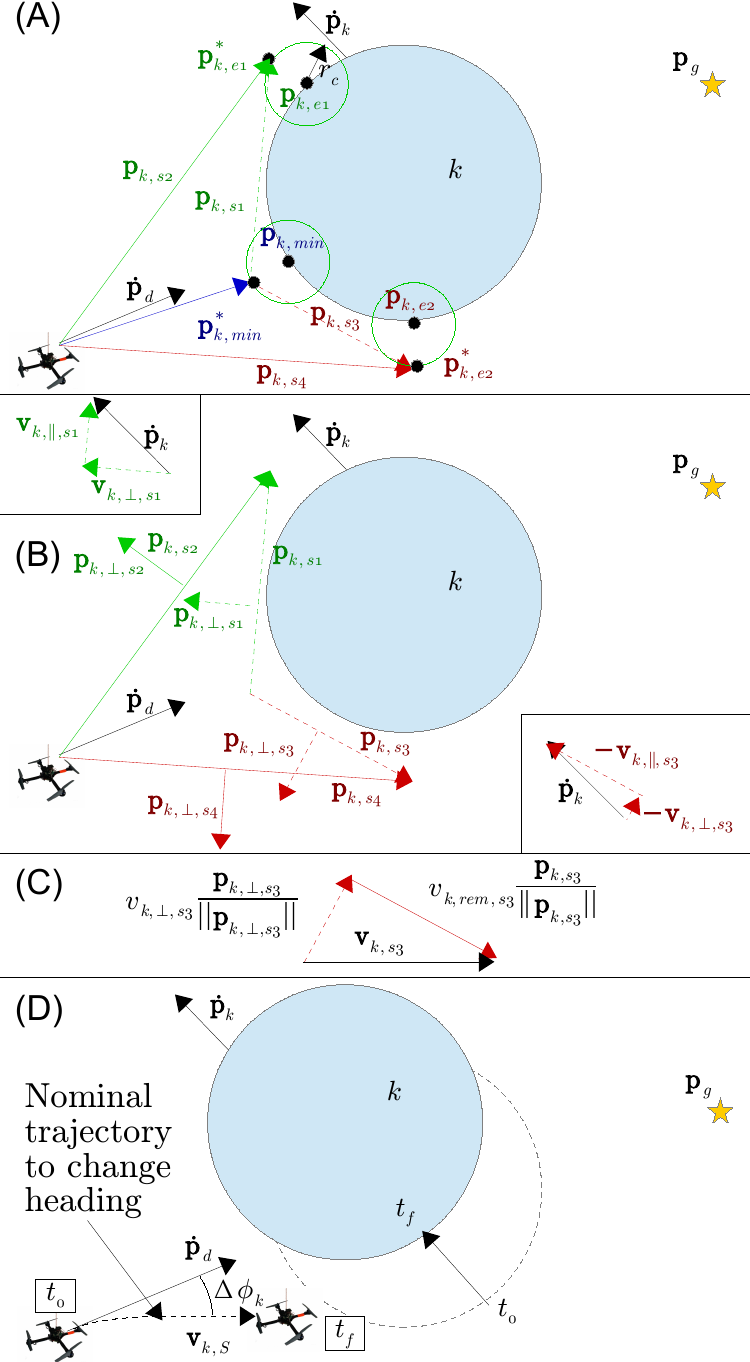}
		\caption{\label{FigUpdatedProcess} Process to determine course change for a moving obstacle, $k$. (A) The vehicle identifies the following points: minimum distance, $\mathbf{p}_{k,min}$, projected minimum distance, $\mathbf{p}_{k,min}^*$, extents, $\mathbf{p}_{k,e1}$ and $\mathbf{p}_{k,e2}$, projected extents, $\mathbf{p}_{k,e1}^*$ and $\mathbf{p}_{k,e2}^*$, and the four candidate tangent directions $\mathbf{p}_{k,si}$ for $i=1,...,4$. (B) The vehicle finds the components of the velocity $v_{k,\perp,si}$ and $v_{k,\parallel,si}$, $i=1,...,4$. The extents of the obstacle prevent navigation to $\mathbf{p}_g$ and since we assume $\mathbf{z}_{k,t}=0$, $\mathbf{p}_{k,e2}$ is chosen as the extent to head towards because it requires less time (Eq.~\ref{EqTtv}) to reach. (C) The desired final velocity is computed for the chosen extent direction. (D) The corresponding course change is implemented and a sample trajectory is shown with the vehicle and obstacle projected forward in time.}
\end{figure}

\subsection{Course Change Definition for a Vehicle}
\label{SecCourseChangeOneVeh}
The process to determine a course change angle and the set of all feasible course angles for a vehicle is very similar to the process described in Sec.~\ref{SecCourseChangeOneObs} for an obstacle. We distinguish critical and non-critical vehicles and also if there are critical obstacles present as there are some slight differences in the calculations.

For non-critical vehicles with critical, non-critical, or no obstacles, and critical vehicles with no critical obstacles, the process is the same as Sec.~\ref{SecCourseChangeOneObs}, except for three modifications: (1) the vehicle uses $r_{c,k}^*$ instead of $r_c$, (2) the vehicle uses Eqs.~\ref{EqPljS1Veh} to \ref{EqPljS4Veh} to define the tangent directions, and (3) since the vehicles are convex $\mathbf{O}_{k,si} = \mathbf{O}_{k,si}'$ for $i = 1,\dots,4$. Using these modifications, the course change angle, $\Delta \phi_k$, is solved from Eq.~\ref{EqDelPhiPerObs}, and the set of all feasible course changes, $\mathbf{O}_k$ is solved from Eq.~\ref{EqSetOrt} for $S$ defined by Eq.~\ref{EqAggresConserv}.

For critical vehicles with critical obstacles present, we include two additional modifications: (1) the vehicle does not fix its circumnavigation direction since the other vehicle can maneuver, and (2) because the circumnavigation is not fixed, it retains the desired course angles and the set of all feasible course angles for both extent points so there are two possible circumnavigation directions. This section describes the process to determine the course angles, $\Delta \phi_{k,\mathbf{z}_I}$ and $\Delta \phi_{k,-\mathbf{z}_I}$, and the corresponding sets of all feasible course angles, $\mathbf{O}_{k,\mathbf{z}_I}$ and $\mathbf{O}_{k,-\mathbf{z}_I}$, associated with each circumnavigation direction, $\mathbf{z}_{I}$ and $-\mathbf{z}_{I}$, for a critical vehicle $k$.

Similar to obstacles, the vehicle first determines the bounding extent and projected extent points to account for the minimum reaction distance, $r_{c,k}^*$. Since there is only a single sensed point (Fig.~\ref{FigPadObsByRC}C), Eqs.~\ref{EqPke1SinglePt} and \ref{EqPke2SinglePt} are used with $r_{c,k}^*$ replacing $r_c$, and we use Eq.~\ref{EqPhiHVeh} instead of Eq.~\ref{EqPhiHObs} for $\phi_h$ since $r_{c,k}^* > r_c$, and it is possible to be within $r_{c,k}^*$ of another vehicle (even though the vehicle should not remain there):
\begin{equation}
	\label{EqPhiHVeh}
	\phi_h = \left\{\begin{array}{ll}
	\cos^{-1} \left(\frac{r_{c,k}^*}{||\mathbf{r}_{k,1}||}\right), & r_{k,1} > r_{c,k}^* \\
	\cos^{-1} \left(\frac{||\mathbf{r}_{k,1}||}{r_{c,k}^*}\right), & r_{k,1} \leq r_{c,k}^*
	\end{array}
	\right.
\end{equation}

\begin{figure}
	\begin{center}
		\includegraphics[width=1.85in]{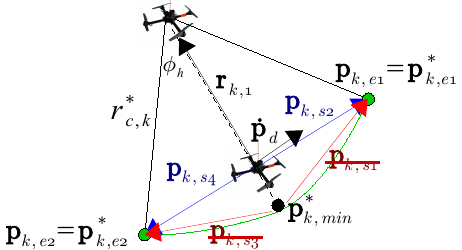}
		\caption{\label{FigPadVehByRC}Extent points and tangent directions for a vehicle within $r_{c,k}^*$ of another vehicle. The conservative tangent directions (red vectors) from  Eqs.~\ref{EqPljS1} to \ref{EqPljS4} are unacceptable because they will not take the vehicle outside $r_{c,k}^*$. Instead, the aggressive tangent directions must be used as defined in Eqs.~\ref{EqPljS1Veh} to \ref{EqPljS4Veh}.}
	\end{center}
\end{figure}
The tangent direction definitions are modified from Eqs.~\ref{EqPljS1} to \ref{EqPljS4} to reduce the time that the vehicle stays within $r_{c,k}^*$. As shown in Fig.~\ref{FigPadVehByRC}, only the aggressive tangent direction is acceptable to navigate the vehicle out of $r_{c,k}^*$ as follows: 
\begin{align}
	\label{EqPljS1Veh}
	\mathbf{p}_{k,s1} &= \left\{\begin{array}{ll}
						\mathbf{p}_{k,e1}^* - \mathbf{p}_{k,min}^*, & ||\mathbf{p}_{k}-\mathbf{p}_d|| \geq r_{c,k}^* \\
						\mathbf{p}_{k,e1}^* - \mathbf{p}_d, & ||\mathbf{p}_{k}-\mathbf{p}_d|| < r_{c,k}^*
						\end{array}
						\right. \\ 
	\mathbf{p}_{k,s2} &= \mathbf{p}_{k,e1}^* - \mathbf{p}_d \\
	\mathbf{p}_{k,s3} &= \left\{\begin{array}{ll}
						\mathbf{p}_{k,e2}^* - \mathbf{p}_{k,min}^*, &  ||\mathbf{p}_{k}-\mathbf{p}_d|| \geq r_{c,k}^*\\
						\mathbf{p}_{k,e2}^* - \mathbf{p}_d, & ||\mathbf{p}_{k}-\mathbf{p}_d|| < r_{c,k}^*
						\end{array}
						\right. \\
	\label{EqPljS4Veh}
	\mathbf{p}_{k,s4} &= \mathbf{p}_{k,e2}^* - \mathbf{p}_d
\end{align}
Next, the vehicle follows the same procedure to identify the velocity components, candidate course changes, and circumnavigation directions as defined in Eqs.~\ref{EqPpar} to \ref{Eqzt}. 

Instead of choosing an extent point for traversing the vehicle, we retain both $\mathbf{p}_{k,e1}$ and $\mathbf{p}_{k,e2}$ to allow flexibility for critical obstacles with fixed circumnavigation directions. For each extent point we select either the conservative or aggressive tangent direction from Eq.~\ref{EqAggresConserv}. 

As a result, we get two sets of all feasible course angles, $\mathbf{O}_{k,\mathbf{z}_I}$ and $\mathbf{O}_{k,-\mathbf{z}_I}$, one for each circumnavigation direction (Eq.~\ref{EqCircumDir}). The course changes (Eq.~\ref{EqDelPhiPerTanDir}) corresponding to each circumnavigation direction are $\Delta \phi_{k,\mathbf{z}_I}$ and $\Delta \phi_{k,-\mathbf{z}_I}$, respectively. 

The course change angles, $\Delta \phi_{k,\mathbf{z}_I}$ and $\Delta \phi_{k,-\mathbf{z}_I}$, and set of all feasible course angles, $\mathbf{O}_{k,\mathbf{z}_I}$ and $\mathbf{O}_{k,-\mathbf{z}_I}$, for vehicle $k$ are used to compare to other obstacles and vehicles to determine a final course change in Sec.~\ref{SecCourseChange}.

\subsection{Course Change Definition for Multiple Obstacles/Vehicles}
\label{SecCourseChange}
To determine the overall course change to safely navigate in the environment, the vehicle uses at minimum the course change angles for the critical obstacles and vehicles ($k\in\mathcal{I}_{co} \cup \mathcal{I}_{cv}$) and may also consider non-critical obstacles. This section describes the process for evaluating and combining the course change angles and sets of all feasible course angles for each obstacle/vehicle to determine a final course change, $\Delta \phi$, and identify the active obstacle.


Starting with the critical obstacles, we take the course change definitions, $\Delta \phi_k$, from Eq.~\ref{EqDelPhiPerObs} and the feasible sets of course angles for obstacles, $\mathbf{O}_{k}$, from Eq.~\ref{EqSetOrtRefined} and combine them for all obstacles $k \in \mathcal{I}_{co}$ to form the set of candidate course changes $\mathbf{O}_{\Delta \phi,o}$ and the feasible set of all course changes $\mathbf{O}_o$, as defined in Eqs.~\ref{EqODelPhiO} and \ref{EqOo}, respectively: 
\begin{align}
	\label{EqODelPhiO}
	\mathbf{O}_{\Delta \phi,o}& = \left\{\Delta \phi_{k}|~k \in \mathcal{I}_{co}\right\} \\
	\label{EqOo}	
	\mathbf{O}_{o} &= \left\{\begin{array}{ll}
	\bigcap_{m=1}^j \mathbf{O}_{k_{m}}, & \begin{subarray}{l}
											\bigcap_{m=1}^{j} \mathbf{O}_{k_m} \neq \emptyset, \\
											\bigcap_{m=1}^{j+1} \mathbf{O}_{k_m} = \emptyset
											\end{subarray} \\
	\underset{k\in \mathcal{I}_{co}}{\bigcap} \mathbf{O}_{k}, & \mathrm{otherwise}
	\end{array}
	\right.
\end{align}
\noindent where the obstacles are evaluated from most to least critical where the most critical obstacle minimizes Eq.~\ref{EqPMinGlobal}. The $k_{j+1}$ obstacle is the first obstacle that produces an empty set of feasible course changes when $\mathbf{O}_{k_j}$ is intersected with all previous sets. By constantly navigating around the closest $j$ obstacles, the vehicle will eventually clear all obstacles as shown in Theorems 2 and 3.


Similarly, if there are only critical vehicles (no critical obstacles), then we get $\mathbf{O}_{\Delta \phi,v}$ and $\mathbf{O}_v$ from Eqs.~\ref{EqODelPhiO} and \ref{EqOo}, respectively, for $k \in \mathcal{I}_{cv}$. 

Recall that when there are both critical obstacles and vehicles we retain both circumnavigation directions so the process is a little different to simplify the course changes and feasible sets. We start by bounding the course changes for each circumnavigation direction as
\begin{align}
	\Delta \phi_{v,\mathbf{z}_I} &=  \underset{k \in \mathcal{I}_{cv}}{\mathrm{min}}(\Delta \phi_{k,\mathbf{z}_I}) \\
	\Delta \phi_{v,-\mathbf{z}_I}& =  \underset{k \in \mathcal{I}_{cv}}{\mathrm{max}}(\Delta \phi_{k,-\mathbf{z}_I})
\end{align}

While the bounding angles provide a course change that navigates around all vehicles, the goal position may still be feasible depending on the vehicle locations. To determine if the goal position is feasible, we define sets of all feasible course change angles for each circumnavigation direction as
\begin{equation}
	\mathbf{O}_{v,\pm\mathbf{z}_I} = \left\{\begin{array}{ll}
	\bigcap_{m=1}^j\mathbf{O}_{k_m,\pm \mathbf{z}_I},& \begin{subarray}{l} \bigcap_{m=1}^{j} \mathbf{O}_{k_m,\pm \mathbf{z}_I} \neq \emptyset \\
	\bigcap_{m=1}^{j+1} \mathbf{O}_{k_m,\pm \mathbf{z}_I} = \emptyset
	\end{subarray} \\
	\underset{k\in \mathcal{I}_{cv}}{\bigcap} \mathbf{O}_{k,\pm \mathbf{z}_I}, ~&\mathrm{otherwise}
	\end{array}
	\right.
\end{equation}

We use the sets for each circumnavigation direction and the bounding course change angles to define the set of vehicle course change angles as
\begin{align}
	\label{EqODelPhiPerVeh}
	\mathbf{O}_{\Delta \phi,v} = \left\{\begin{array}{ll}
	\left\{\Delta \phi_g\right\}, & \Delta \phi_g \in \mathbf{O}_{v,\mathbf{z}_I} \cap \mathbf{O}_{v,-\mathbf{z}_I} \\
	\left\{\Delta \phi_g,\Delta \phi_{v,-\mathbf{z}_I}\right\}, & \Delta \phi_g \in \mathbf{O}_{v,\mathbf{z}_I}, \Delta \phi_g \notin \mathbf{O}_{v,-\mathbf{z}_I} \\
	\left\{\Delta \phi_{v,\mathbf{z}_I},\Delta \phi_g\right\}, & \Delta \phi_g \in \mathbf{O}_{v,-\mathbf{z}_I},\Delta \phi_g \notin \mathbf{O}_{v,\mathbf{z}_I} \\
	\left\{\Delta \phi_{v,\mathbf{z}_I},\Delta \phi_{v,-\mathbf{z}_I}\right\}, &\mathrm{otherwise}
	\end{array}
	\right.
\end{align}

Now that the critical obstacles and vehicles have been combined independently, a final course change must be determined. As the vehicle evaluates the candidate course changes and sets of all feasible course changes, it does so according to the following principles. The main objective is to reach the goal position, so the ability to navigate towards the goal position is checked first, as given by condition 1 in Eq.~\ref{EqFinalDelPhi}. 

Next, we consider three cases where navigation to the goal position is not feasible and there are critical obstacles and/or critical vehicles. If there are critical obstacles and vehicles, and there is no viable course change, which occurs when $\mathbf{O}_{\Delta \phi,v} \cap \mathbf{O}_o = \emptyset$, then re-prioritization is necessary. In this case, the vehicle chooses a course change in $\mathbf{O}_o$ that is closest to $\Delta \phi_g$ and assigns itself a higher priority (relative to the other vehicles in $\mathcal{I}_{cv}$) until it clears $r_{c,k}^*$. This is condition 2 in Eq.~\ref{EqFinalDelPhi}.

If there are critical obstacles and vehicles but $\mathbf{O}_{\Delta \phi,v} \cap \mathbf{O}_o \neq \emptyset$, then the vehicle chooses the desired course change from $\mathbf{O}_{\Delta \phi,v}$ that is in $\mathbf{O}_o$. This is condition 3 of Eq.~\ref{EqFinalDelPhi}. 

If there are only critical obstacles, then the vehicle chooses the desired course change from $\mathbf{O}_{\Delta \phi,o}$ that is in $\mathbf{O}_o$ or the feasible course change from $\mathbf{O}_o$ that is closest to one of the desired changes in $\mathbf{O}_{\Delta \phi,o}$. This is condition 4 of Eq.~\ref{EqFinalDelPhi}.

If there are only critical vehicles, then the choice is the same as the case where there are only critical obstacles, except we use $\mathbf{O}_{\Delta \phi,v}$ and $\mathbf{O}_v$. This is condition 5 of Eq.~\ref{EqFinalDelPhi}.

In the event that all the critical obstacles allow the goal position or there are no critical obstacles, then the non-critical obstacles are evaluated in increasing order starting with the most imminent (i.e.~the one that minimizes Eq.~\ref{EqPMinGlobal}). The first non-critical obstacle/vehicle $k_j$ where $\Delta \phi_g \notin \bigcap_{i=1}^j \mathbf{O}_{k_i}$ for $k_i \in \mathcal{I}_{no} \cup \mathcal{I}_{nv}$ defines the desired course change $\Delta \phi_{k_j}$. The desired course change must respect the sets of all feasible course angles for all preceding obstacles/vehicles. This is condition 6 in Eq.~\ref{EqFinalDelPhi}.
If none of the obstacles/vehicles prohibit navigation to the goal position then $\Delta \phi = \Delta \phi_g$ which is condition 7 of Eq.~\ref{EqFinalDelPhi}. These seven conditions, summarized in Eq.~\ref{EqFinalDelPhi}, are evaluated in sequence until a condition is met:
\begin{align}
	\label{EqFinalDelPhi}
	\Delta \phi = \left\{\begin{array}{ll}
	\Delta \phi_g, & ||\mathbf{p}_g - \mathbf{p}_d|| < ||\mathbf{p}_{min}-\mathbf{p}_d|| \\
	\underset{\begin{subarray}{c}
		\Delta \phi \in \mathbf{O}_{o}
		\end{subarray}}{\mathrm{argmin}}\left(|\Delta \phi - \Delta \phi_{g}|\right), & \mathbf{O}_{c,z} = \emptyset, \mathcal{I}_{co} \neq \emptyset, \mathcal{I}_{cv} \neq \emptyset \\
	\underset{\begin{subarray}{c}
		\Delta \phi \in \mathbf{O}_{o} \\
		\Delta \phi_{v} \in \mathbf{O}_{\Delta \phi,v}
		\end{subarray}}{\mathrm{argmin}}\left(|\Delta \phi - \Delta \phi_{v}|\right), & \Delta \phi_g \notin \mathbf{O}_{c,z}, \mathcal{I}_{co} \neq \emptyset, \mathcal{I}_{cv} \neq \emptyset \\
		\underset{\begin{subarray}{c}
		\Delta \phi \in \mathbf{O}_{o} \\
		\Delta \phi_{o} \in \mathbf{O}_{\Delta \phi,o}
		\end{subarray}}{\mathrm{argmin}}\left(|\Delta \phi - \Delta \phi_{o}|\right), & \Delta \phi_g \notin \mathbf{O}_{o}, \mathcal{I}_{co} \neq \emptyset,\mathcal{I}_{cv} = \emptyset \\
		\underset{\begin{subarray}{c}
		\Delta \phi \in \mathbf{O}_{v} \\
		\Delta \phi_{v} \in \mathbf{O}_{\Delta \phi,v}
		\end{subarray}}{\mathrm{argmin}}\left(|\Delta \phi - \Delta \phi_{v}|\right), & \Delta \phi_g \notin \mathbf{O}_{v}, \mathcal{I}_{co} = \emptyset,\mathcal{I}_{cv} \neq \emptyset \\
		\underset{\begin{subarray}{c}
		\Delta \phi \in \mathbf{O}_{c,e}
		\end{subarray}}{\mathrm{argmin}}\left(|\Delta \phi - \Delta \phi_{k_j}|\right), & \Delta \phi_g \in \mathbf{O}_{c,e}^- \\
		\Delta \phi_g, & \Delta \phi_g \in \mathbf{O}_{c,e}
		\end{array} 
		\right. 
\end{align}
\noindent where
\begin{equation}
	\mathbf{O}_{c,e} = \left\{\begin{array}{ll}
	\mathbf{O}_o \cap \left(\bigcap_{i=1}^j \mathbf{O}_{k_i}\right), & \mathcal{I}_{co} \neq \emptyset, \mathcal{I}_{cv} = \emptyset \\
	\mathbf{O}_{v} \cap \left(\bigcap_{i=1}^j \mathbf{O}_{k_i}\right), & \mathcal{I}_{co} = \emptyset, \mathcal{I}_{cv} \neq \emptyset \\
	\mathbf{O}_{c,z} \cap \left(\bigcap_{i=1}^j \mathbf{O}_{k_i}\right), & \mathcal{I}_{co} \neq \emptyset, \mathcal{I}_{cv} \neq \emptyset \\
	\bigcap_{i=1}^j \mathbf{O}_{k_i}, & \mathcal{I}_{co} = \emptyset, \mathcal{I}_{cv} = \emptyset 
	\end{array}
	\right.
\end{equation}
\noindent where $z$ is either $\pm \mathbf{z}_I$ depending on which circumnavigation direction has been chosen, $\mathbf{O}_{c,z} = \mathbf{O}_o \cap \mathbf{O}_{v,z}$ for $z = \pm \mathbf{z}_I$, and $\mathbf{O}_{c,e}^-$ is the same as $\mathbf{O}_{c,e}$ but evaluated up to non-critical obstacle $k_{j-1}$ (i.e.~the last non-critical obstacle/vehicle that still allows navigation to the goal position).

The active obstacle is the obstacle $k_a$ that minimizes Eq.~\ref{EqPMinGlobal} and satisfies $\Delta \phi_g \notin \mathbf{O}_{k_a}$. The circumnavigation direction is only set for obstacles, and remains fixed until $||\mathbf{p}_{k_a,min}-\mathbf{p}_d|| > r_s$. The course change defined by Eq.~\ref{EqFinalDelPhi} is used to generate smooth trajectories in Sec.~\ref{SecSigmoid}. 

\subsection{Velocity Change Definition}
\label{SubSecVelChange}
The vehicle can also make velocity changes to avoid collisions and safely navigate the environment. This section examines the cases where velocity change is appropriate and the process to determine the velocity change, $\Delta v$. 

The only adjustments in velocity are in the three cases described in this section: (1) slowing to reach the goal position, (2) when the vehicle cannot safely pass another vehicle (due to the presence of obstacles), and (3) when the vehicle has come within $r_{c,k}^*$ of another vehicle. We briefly examine these cases.

First, when there is a clear path to the goal position and the vehicle is within some critical distance, $r_{goal}$, of it, the vehicle computes a trajectory to come to a stop at the goal position (or in the case of a fixed wing vehicle, to come to a loiter). The critical distance may be the minimum stopping distance, or a user-defined distance that is greater than the minimum stopping distance. This is the first condition in Eq.~\ref{EqFinalDelV}.

Second, if a vehicle detects another vehicle within its sensor range and there are also obstacles within sensor range of one or both vehicles, the vehicle slows to match the velocity of the slowest vehicle within sensor range, or connected to a vehicle within sensor range. Since the obstacle spacing $r_{obs} \geq 2 r_c + r_{180}$ (from Assumption \ref{AssumpLessCapable}), there is no guarantee that multiple vehicles can fit between obstacles, or that the more capable vehicle is able to safely complete a passing maneuver. The safe solution is then to match velocity so all vehicles can maneuver safely. This is the second condition in Eq.~\ref{EqFinalDelV}.

Lastly, we consider the case where the vehicle has already matched the velocity of a slower vehicle, but the slower vehicle has maneuvered so that $r_{c,k}^*$ is violated. The maneuvering vehicle can either change course (as already established in Sec.~\ref{SecCourseChange}) or it can slow down temporarily. To determine if maintaining course and decreasing velocity is appropriate, we use $\mathbf{v}_{k,s1}$, which is computed in Eq.~\ref{EqVfDef} for the conservative tangent direction. If the course change associated with $\mathbf{v}_{k,s1}$ and a decrease in velocity with $\Delta \phi = 0$ both produce relative velocity vectors in the same direction (e.g.~$\mathbf{v}_{rel} = -\mathbf{\dot{p}}_{d,k} + \mathbf{v}_{k,s1}$), then the vehicle decreases velocity according to condition 3 of Eq.~\ref{EqFinalDelV} instead of performing a lengthy course change maneuver. Since the vehicle is slowing down (instead of maneuvering), it considers course changes for the next closest obstacle (as defined in Sec.~\ref{SecCourseChange}, Eq.~\ref{EqFinalDelPhi}). 
Once the vehicle has re-established a distance $>r_{c,k}^*$ from the other vehicle, then it resumes its previous velocity. This is the final condition in Eq.~\ref{EqFinalDelV}. These four conditions, summarized in Eq.~\ref{EqFinalDelV}, are evaluated in sequence until a condition is met:
\begin{align}
	\label{EqFinalDelV}
	\Delta v = 	\left \{\begin{array}{ll}
-||\mathbf{\dot{p}}_d||, & ||\mathbf{p}_g - \mathbf{p}_d|| \leq r_{goal} \\
	\min \limits_{k \in \mathcal{I}_{nr}} \Delta v_{k}, &  \mathcal{I}_{obs,k} \neq \emptyset, \mathcal{I}_{obs} \neq \emptyset   \ \\
	\left(K_{\Delta v}-1\right) ||\mathbf{\dot{p}}_{d}||,~~& r_{k} < r_{c,k}^*, \frac{v_{min}}{||\mathbf{\dot{p}}_{d}||}\leq K_{\Delta v} <1, 0 \in \mathbf{O}_o \\
	v_c - ||\mathbf{\dot{p}}_d||, & \mathrm{otherwise} 
	\end{array}
	\right.
\end{align}
\noindent where $\Delta v_{k} = ||\mathbf{\dot{p}}_{k}|| - ||\mathbf{\dot{p}}_d||, \forall ||\mathbf{\dot{p}}_{k}|| > 0$, $r_{k} = ||\mathbf{p}_{k} - \mathbf{p}_d||$, and $K_{\Delta v}$ is a component of the solution to $[\mathbf{\dot{p}}_{d},~-\mathbf{v}_{rel}]\mathbf{K} = \mathbf{\dot{p}}_{d,k}$, where $\mathbf{K} = [K_{\Delta v},~K_{v_{rel}}]^T$ and $K_{\Delta v}$ and $K_{v_{rel}}$ are scaling constants for the $\mathbf{\dot{p}}_d$ and $\mathbf{v}_{rel}$ vectors, respectively, so that $-\mathbf{\dot{p}}_{d,k} + K_{\Delta v} \mathbf{\dot{p}}_d = K_{v_{rel}} \mathbf{v}_{rel}$ is satisfied. If the scaling constant, $K_{\Delta v}$, causes the vehicle to go below its minimum velocity, $v_{min}$, (e.g.~ a fixed wing aircraft) then the vehicle maintains velocity and makes a course change instead. The velocity change defined by Eq.~\ref{EqFinalDelV} is used to generate smooth trajectories in Sec.~\ref{SecSigmoid}.



\subsection{Smooth course and velocity transitions}
\label{SecSigmoid}
The trajectory generation algorithm utilizes sigmoid functions to transition from the previous course, $\phi_{n-1}$, and velocity, $v_{n-1}$, to a new course, $\phi_{n}$, and velocity, $v_{n}$, as determined by the course and velocity changes from Eqs.~\ref{EqFinalDelPhi} and \ref{EqFinalDelV}. This section describes the process to make the course and velocity transitions by defining the desired trajectory, $\mathbf{p}_d(t)$ and $\mathbf{\dot{p}}_d(t)$, for $t_{o,n} \leq t \leq t_{o,n} + \tau_{f,n}$, where $t_{o,n}$ is the start of the $n^{th}$ sigmoid curve and $\tau_{f,n}$ is the $n^{th}$ sigmoid curve timespan. 

The hyperbolic tangent function ($\tanh$) is chosen because of its widespread use in generating smooth motion transitions \cite{Fischer2014}. We define the course and velocity functions and their first derivatives as
\begin{align}
			\label{EqSigPhi} \phi &= c_1 \tanh(c_2 \tau - c_3) + c_4 \\
			\dot{\phi}& = c_1 c_2 \left(1-\tanh^2(c_2 \tau - c_3) \right) \\
			\label{EqSigV} v &= d_1 \tanh(d_2 \tau - d_3) + d_4 \\
			\label{EqSigVdot} \dot{v} &= d_1 d_2 \left(1-\tanh^2(d_2 \tau - d_3)\right)
\end{align}

\noindent where $c_i$ and $d_i$ are coefficients to be determined and $\tau$ is the sigmoid curve time. The desired velocity vector is
\begin{equation}
	\label{EqDotPdVec}
	\mathbf{\dot{p}}_d = \left[\begin{array}{c}
					v \cos \phi \\
					v \sin \phi
					\end{array}
					\right]
\end{equation} 
The coefficients are solved analytically by considering the following assumptions: (1) each sigmoid function occurs over the time interval $\tau=0$ to $\tau=\tau_f$, and (2) since $\tanh$ asymptotically approaches -1 and 1, the bounds of the function are approximated by $\pm(1-\varepsilon)$, (where we use $\varepsilon = 10^{-3}$ to reduce the error of this approximation to $<1$\%). The coefficient solutions are written as
\begin{align}
	\label{EqCoeffSummary1}
	c_3 &= d_3 = \tanh^{-1}(1-\varepsilon) = 3.8\\
	\label{EqCoeffSummaryc2}
	c_2 &= d_2 = 2 c_3/\tau_{f} = 7.6/\tau_{f} \\
	\label{EqCoeffSummaryc1}
	c_1 &= c_4 = 0.5 \Delta \phi \\
	\label{EqCoeffSummaryEnd}
	d_1 &= d_4 = 0.5 \Delta v 
\end{align}
\noindent where the sigmoid curve timespan, $\tau_f$, is defined later in Theorem \ref{ThTauF} to respect the vehicle thrust limitations. 

As the vehicle navigates the environment, it receives new sensor information every $\Delta T_s$ seconds. If the sensor update rate is very fast, the vehicle likely does not complete the desired course and velocity changes before new sensor data is available. As a result, the vehicle must wait until there is available thrust then compute course and velocity changes from the most recent sensor data. 

The maximum delay in starting the next maneuver in the cruise velocity constraints is defined in Sec.~\ref{SecSafetyGuarantee}B. This delay allow subsequent maneuvers to begin before previous ones are finished without violating the maximum thrust constraint.

To make use of the new sensor information, the vehicle sums successive sigmoid curves as shown in Fig.~\ref{FigSigIncv2}. Since the sigmoid function and its first four derivatives asymptotically approach 0 (effectively are 0) at $\tau = 0$ and $\tau = \tau_f$, the summed sigmoid curve provides the smoothness guarantee of Property 1. 
\begin{figure}
	\centering
		\includegraphics[width=2.95in]{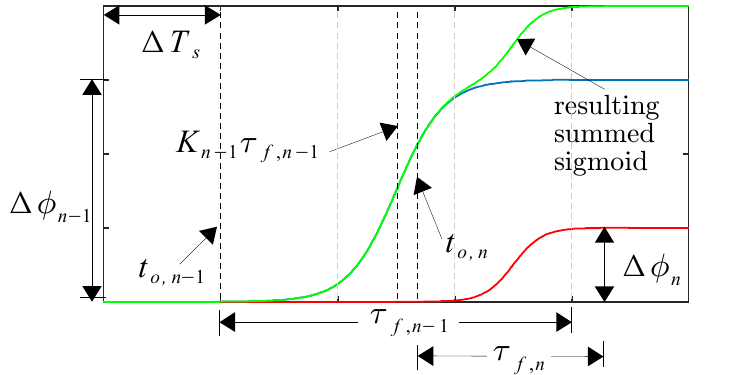}
		\caption{\label{FigSigIncv2} Example of summed sigmoid curves. At $t_{n-1}$ the vehicle determines a course change of $\Delta \phi_{n-1}$ (blue curve) and starts the sigmoid at time $t_{o,n-1}$. With new sensor information an additional course change of $\Delta \phi_n$ (red curve) is necessary. This curve cannot start immediately because it would violate the maximum thrust. Instead, the vehicle computes $t_{o,n}$ from Eq.~\ref{EqTimeOffset} which is after the $n-1$ curve peak acceleration at $K_{n-1}\tau_{f,n-1}$. The two curves are summed to produce a smooth course change (green curve) that does not violate $f_{max}$.}
\end{figure}
The sigmoid curves cannot be summed arbitrarily without violating the vehicle's maximum thrust, $f_{max}$. Instead, we scale and shift subsequent maneuvers such that the summation of the maneuvers (i.e.~sigmoid curves) is bounded to respect $f_{max}$. The curve scaling is achieved by varying the sigmoid curve timespan, $\tau_f$, and we introduce an offset time, $t_{o,n}$, to shift the curve start. The summed sigmoid functions are defined as
\begin{align}
	\label{EqSigPhiSum} \phi(t) &= \sum_{i=1}^n \left(c_{1,i} \tanh(c_{2,i} (t_n - t_{o,i}) - c_{3,i}) + c_{4,i}\right) \\
	\label{EqSigVSum} v(t) &= \sum_{i=1}^n \left(d_{1,i} \tanh(d_{2,i} (t_n- t_{o,i}) - d_{3,i}) + d_{4,i}\right) 
\end{align}
\noindent where
\begin{equation}
	\label{EqTimeOffset}
	t_{o,n} = \left\{\begin{array}{ll}
		t_n, &n=1 \\		
		\mathrm{max}(t_n+\Delta T_c,t_{o,n-1} + t_{int,n}),& n>1
		\end{array}
		\right.
\end{equation}
\noindent where $t_n \geq 0$ is the current time and $t_{int,n}$ is the point on the previous sigmoid curve where thrust becomes available for the next maneuver. 

The definition for $t_{int,n}$ is the intersection point between the sigmoid acceleration curve ($\mathbf{\ddot{p}}_d(t)$), and a linear approximation of the sigmoid curve acceleration as shown in Fig.~\ref{FigSummedSigExplain}A and defined as
\begin{align}
	\label{EqTintIneq}
	& K_{n-1}\tau_{f,n-1} < t_{int,n} < \tau_{f,n-1} \\
	\label{EqTintDef}
	&\frac{-a_{max}}{(1-K_{n-1})\tau_{f,n-1}}t_{int,n} + \frac{a_{max}}{1-K_{n-1}} = \sqrt{v_{n-1}(t_{int,n})^2 \dot{\phi}_{n-1}(t_{int,n})^2 + \dot{v}_{n-1}(t_{int,n})^2}
\end{align}
\noindent We use the linear approximation because it provides simple upper bounds on the sum of two successive $||\mathbf{\ddot{p}}_d(t)||$ curves when the second curve starts at $t_{o,n}$ (as proven in Appendix A). 

The linear approximation slope for the rising, $h_n^+$, and falling, $h_n^-$, sides of the sigmoid curve acceleration are defined by
\begin{align}
	\label{EqSlopePlus}
	h_n^+ &= \frac{a_{max}}{K_n \tau_{f,n}} \\
	\label{EqSlopeMinus}
	h_n^- &= \frac{a_{max}}{(1-K_n) \tau_{f,n}}
\end{align}
\noindent where $a_{max}$ is the remaining acceleration available for tracking the trajectory (after lift and drag forces have been accounted for) and defined in Theorem \ref{ThTauF} by Eq.~\ref{EqAmaxDef}, and $K_n \tau_{f,n}$ is the sigmoid curve time where the acceleration is maximum as defined by
\begin{align}
	\label{EqTnMaxFromKn}
	\tau_{n,max} &= K_n \tau_{f,n},~\mathrm{where}~0\leq K_n\leq 1 \\
	\label{EqKDef}
	K_n &= \frac{1}{2}\left(\frac{\tanh^{-1}(H_n)}{c_3}+1\right) \\
	\label{EqHDef}
	H_n &= \tanh(c_{2,n}\tau_{n,max} -  c_{3,n})
\end{align}
The computation of $H_n$ is described in Theorem 1 and its proof. We use the linear approximation slope terms when solving for $\tau_{f,n}$ and $t_{o,n}$ to match slopes and respect $a_{max}$ as shown in Fig.~\ref{FigSummedSigExplain}B.

Once the sigmoid curve timespan, $\tau_{f,n}$, is solved according to Theorem 1, the trajectory is fully defined by Eq.~\ref{EqDotPdVec} which uses the definitions for $\phi(t)$ and $v(t)$ from Eqs.~\ref{EqSigPhiSum} and \ref{EqSigVSum}. The trajectory is defined over $t_{o,n} \leq t \leq t_{o,n} + \tau_{f,n}$ where $t_{o,n}$ is defined by Eq.~\ref{EqTimeOffset}. A vehicle controller, such as the one described in Sec.~\ref{SecVehController}, follows the trajectory to navigate the vehicle. 


 
\begin{figure}
	\centering
		\includegraphics[width=2.95in]{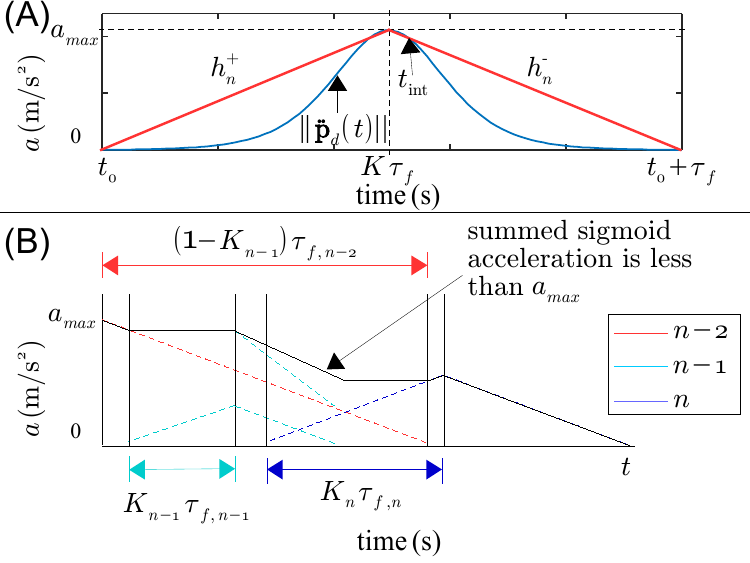}
		\caption{\label{FigSummedSigExplain} (A) The estimated sigmoid curve slopes, $h_n^+$ and $h_n^-$, for a single sigmoid acceleration curve are approximated linearly. The acceleration curve is not always symmetric, e.g. when both $\Delta v \neq 0$ and $\Delta \phi \neq 0$ so $h_n^+$ does not always equal $h_n^-$. (B) Three sigmoid functions are summed together, where the estimated slopes of curves $n-1$ and $n$ match the slope of $n-2$. The resulting summed sigmoid acceleration does not violate $a_{max}$. }
	
\end{figure}


\section{Trajectory Guarantees}
\label{SecSafetyGuarantee}
To guarantee the vehicle can navigate safely in the environment, we present several theorems that define the sigmoid curve timespan, bound the maximum velocity, guarantee that the vehicle clears obstacles and other vehicles by $r_c$, and guarantee that the vehicle reaches the goal position in finite time. The proofs for each of these theorems are given as Appendices A-C, respectively.

To aid in the theorems and proofs we define several quantities. First, the maximum available planar force, $f_{p,max}$, and drag force, $\mathbf{f}_w$, are defined by
\begin{align}
\label{EqFplanarMaxDef}
	f_{p,max} &= \sqrt{f_{max}^2 - (mg)^2} \\
	\label{EqFdrag}
	\mathbf{f}_{w} &= K_d ||\mathbf{v}_{w}||^2 (-\mathbf{x}_W) \\
	\label{EqKdDefinition}
	K_d &= \frac{1}{2} \rho C_{D} A_{x_W}
\end{align}
\noindent where $m$ is the vehicle mass, $g$ is gravity, $\mathbf{v}_w = \mathbf{\dot{p}} - \mathbf{v}_{air}$ is the relative wind velocity between the vehicle and the air, $\mathbf{x}_W$ is the wind frame axis aligned with $\mathbf{v}_w$, $\rho$ is the air density, $C_{D}$ is the coefficient of drag, and $A_{x_W}$ is the cross sectional area normal to $\mathbf{v}_w$.

Additionally, we define the planar force vector as the sum of the trajectory and drag forces 
\begin{equation}
	\label{EqFplanarNormalDrag}
	\mathbf{f}_{planar} = m\mathbf{\ddot{p}}_d + \mathbf{f}_w
\end{equation}
\noindent where the maximum planar force magnitude, $f_{p,max}$, occurs when the vectors are aligned. Considering the two components independently, the maximum drag is $K_d v_{w,max}^2$, where $v_{w,max}$ is the maximum wind speed defined in the theorems, and the desired acceleration magnitude is obtained from Eq.~\ref{EqDotPdVec} as
\begin{equation}
	\label{EqAccelMaxMag}
	||\mathbf{\ddot{p}}_d|| = \sqrt{v^2 \dot{\phi}^2 + \dot{v}^2}
\end{equation}
We can further manipulate Eq.~\ref{EqAccelMaxMag} to define its maximum value by substituting the definitions for $K$ and $H$ from Eqs.~\ref{EqKDef} and \ref{EqHDef} into the sigmoid functions from Eqs.~\ref{EqSigPhi} and \ref{EqSigV}. If we also utilize Eq.~\ref{EqCoeffSummaryc2},  we can isolate the dependency on $\tau_f$ as
\begin{equation}
	\label{EqAccelMaxMagSigVars}
	||\mathbf{\ddot{p}}_{d,max}||^2 = \left(\frac{2 c_3}{\tau_f}\right)^2 S_{traj}
\end{equation}
\noindent where 
\begin{equation}
	\label{EqStrajDef}
	S_{traj} = (c_1(d_1 H + d_4)(1-H^2))^2 + (d_1(1-H^2))^2
\end{equation}
\noindent and the solution for $H$ is derived in Appendix A. Each of these relationships is referenced in the following theorems.

\subsection{Sigmoid Curve Timespan}
The sigmoid curve timespan is bound by Theorem \ref{ThTauF} to ensure that the sigmoid curve does not violate $f_{max}$. Theorem \ref{ThTauF} also defines the offset time for the sigmoid curve start.
\begin{theorem}
\label{ThTauF}
Let the sigmoid curve timespan, $\tau_{f}$, for the $n^{th}$ sigmoid be defined as:
\begin{equation}
\label{EqTauFTh}
\tau_{f,n} \geq \left\{\begin{array}{ll}
\tau_{f,min,n}, & n = 1 \\
\tau_{f,min,n}, & t_n \geq \max\limits_{i<n}(t_{o,i} + \tau_{f,i}) \\
\tau_{f,min,n}, & \frac{(1-K_{n-1})\tau_{f,n-1}}{(K_{n})\tau_{f,min,n}} \leq 1 \\
\sqrt{\frac{2c_3 \sqrt{S_{traj}}}{K_{min}h_n}} ,& \mathrm{otherwise}
\end{array}
\right.
\end{equation}
\noindent where
\begin{align}
	\label{EqTauFMinDef}
	\tau_{f,min,n}& = \frac{2 c_3}{a_{max}} \sqrt{S_{traj}} \\
	\label{EqSigSlope}
	h_n &= \left\{\begin{array}{ll}
h_n^-, & n = 1 \\
h_n^-, & t_n \geq \max\limits_{i<n}(t_{o,i} + \tau_{f,i}) \\
h_n^-, & \frac{(1-K_{n-1})\tau_{f,n-1}}{K_{n}\tau_{f,min,n}} \leq 1 \\
h_{n-1}^-,& \mathrm{otherwise}
\end{array}
\right. \\
	\label{EqAmaxDef}
	a_{max} &= 1/m\left(f_{p,max} - K_d v_{w,max}^2 \right)
\end{align}
\noindent $v_{w,max} = \mathrm{max}(v_{n-1},v_{n-1} + \Delta v_n) + v_{air}$, $v_{n-1}$ is the desired final velocity at $t_{n-1}$, $K_{min}=\mathrm{min}(K_n,1-K_n)$, and $H_n$ is the real solution to
\begin{equation}
	\label{EqHCubic}
  -3c_1^2 d_1^2 H_n^3 - 5 c_1^2d_1 d_4 H_n^2 +
  (-2 c_1^2 d_4^2 + d_1^2 c_1^2 - 2 d_1^2) H_n + d_1 d_4 c_1^2 = 0
\end{equation}
\noindent that satisfies $|H_n| < 1-\varepsilon$. Then, if $\tau_{f,n}$ is always chosen to satisfy Eq.~\ref{EqTauFTh}, the vehicle trajectory does not violate $f_{max}$ in the presence of a bounded wind disturbance velocity, $v_{air,max}$, that satisfies $v_{air,max} < \sqrt{f_{planar}/K_d}$.
\end{theorem}

\begin{proof}
See Appendix A.
\end{proof}
\subsection{Cruise Velocity Bound}
Theorem \ref{ThVc} derives an upper bound on the vehicle's cruise velocity based on its maximum thrust, sensor update rate and range, and obstacle spacing. These criteria result in three inequalities that must be satisfied for the cruise velocity. Since each inequality provides an upper bound on the velocity, the minimum is chosen.
\begin{theorem}
\label{ThVc}
Let the vehicle's maximum cruise velocity be defined as
\begin{equation}
\label{EqThvc}
v_c \leq \mathrm{min}\left(v_{c,f},~v_{c,s},~v_{c,o}\right)
\end{equation}
\noindent where $v_{c,f}$ is the minimum real, positive solution of 
\begin{equation}
\label{EqFplanarInequalityTh}
	\left(\frac{m}{r_{min}} + K_d\right) v_{c,f}^2 + 2K_d v_{air,max} v_{c,f} + v_{air,max}^2 - f_{planar} = 0
\end{equation}
\noindent $v_{c,s}$ satisfies the following inequalities:
\begin{align}
	\label{EqVcSensorIneq}
	v_{c,s} &\leq \frac{r_s - (t_{d,max}+\tau_{f,s}) v_{o,max} - r_c}{t_{d,max}+ \int_0^{\tau_{f,s}} \sin \phi(t) dt} \\
	\label{EqVcSensorTauF}
	\tau_{f,s} &\geq 
	\frac{c_3}{a_{max}} \left(\tan^{-1}\left(\frac{v_{o,max}}{\sqrt{v_{c,s}^2 - v_{o,max}^2}}\right) + \frac{\pi}{2}\right) v_{c,s}
\end{align}
\noindent where
\begin{equation}
	\label{EqTdmax}
	t_{d,max} = \left\{\begin{array}{ll}
	\frac{0.54c_3\pi}{a_{max}}v_{c,s} + \Delta T_c, & 0.54 \tau_{f,s} \geq 2\Delta T_s\\
	2 \Delta T_s + \Delta T_c, & \mathrm{otherwise}
	\end{array}
	\right.
\end{equation}
\noindent and $v_{o,max}$ is the maximum expected obstacle velocity in the environment (or $v_{o,max}=v_{c,s}$ if the maximum is unknown, producing $\tau_{f,s} \geq \frac{c_3}{a_{max}}\pi v_{c,s}$), $\Delta T_c$ is the maximum on-board algorithm computation time, and $v_{c,o}$ satisfies the following inequalities:
\begin{align}
	\label{EqVcObsIneq}
	v_{c,o} &\leq \frac{r_{obs} - 2r_c}{t_{d,max}+\int_0^{\tau_{f,o}} \sin \phi(t) dt} \\
	\label{EqVcObsTauF}
	\tau_{f,o} &\geq \frac{ c_3}{a_{max}}\frac{1}{2} \pi v_{c,o}
\end{align}
\noindent where $r_{obs}$ is the minimum distance between two obstacles. If the vehicle's maximum cruise velocity satisfies Eq.~\ref{EqThvc}, then the vehicle does not violate $f_{max}$ when making a turn of radius $ r_t \geq r_{min}$ in the presence of a bounded wind velocity disturbance that satisfies $v_{air,max} < \sqrt{f_{planar}/K_d}$ and safely clears obstacles by $r_c$.
\end{theorem}

\begin{proof}
See Appendix B.
\end{proof}

\subsection{Goal Position Convergence and Clearance Radius Guarantee}
The convergence to the goal position and clearance radius guarantee are defined in Theorem \ref{ThGoal} to ensure that the vehicle reaches the goal position in finite time, avoids collisions, and clears all obstacles and vehicles by $r_c$.
\begin{theorem}
\label{ThGoal}
Let the course change, velocity change, and circumnavigation direction be defined by Eqs.~\ref{EqFinalDelPhi}, \ref{EqFinalDelV}, and \ref{EqCircumDir}, respectively. If the vehicle generates a trajectory from these equations, and the maximum cruise velocity satisfies Theorem \ref{ThVc}, then the vehicle reaches the goal position (provided $||\mathbf{\dot{p}}_g||=0$) in finite time and clears all obstacles by $r_c$. 
\end{theorem}

\begin{proof}
See Appendix C.
\end{proof} 

\section{Vehicle and Controller}
\label{SecVehController}
The vehicle dynamics for a quadrotor are given in Eqs.~\ref{EqEOMLin} and \ref{EqEOMAng}. Equation \ref{EqEOMLin} is written in the inertial frame, and Eq.~\ref{EqEOMAng} is written in the body frame: 
\begin{align}
	\label{EqEOMLin}
	m\mathbf{\ddot{p}} &= \mathbf{f} + m\mathbf{g} + \mathbf{d}_{p} \\
	\label{EqEOMAng}
	\mathbf{J}\mathbf{\dot{\omega}} &= \mathbf{\omega} \times \mathbf{J} \mathbf{\omega} + \mathbf{u} + \mathbf{R}_{IB}\mathbf{d}_{\omega}
\end{align}

\noindent where $\mathbf{f}$ is the total thrust, $\mathbf{d}_p$ is the translational disturbance (including drag), $\mathbf{J}$ is the vehicle moment of inertia, $\mathbf{\dot{\omega}}$ is the rotational acceleration, $\mathbf{u}$ is the total torque, $\mathbf{R}_{IB}$ is the rotation matrix from the inertial to body frame, and $\mathbf{d}_{\omega}$ is the rotational disturbance. The control inputs are the vehicle force, $\mathbf{f}$, and torque, $\mathbf{u}$. 

The vehicle dynamics also include aerodynamic effects on the propellers like thrust reduction from propeller inflow velocity \cite{Leishman2006,Bramwell,Newman1994} and blade flapping \cite{Hoffmann}. For the purposes of the control law, these terms are added to the disturbance term.

The vehicle controller uses an inner-, and outer-loop control similar to \cite{Bialy2013} and \cite{Cao2016}, where the outer loop controls the translational component and the inner loop controls the rotational component. The outer loop uses a nonlinear robust integral of the sign of the error (RISE) controller \cite{FischerNL2014}, summarized in Eqs.~\ref{EqFischerControl} to \ref{EqFischere2}. The inner loop utilizes the PID control in Eq.~\ref{EqPIDControl} \cite{Cao2016}:
\begin{align}
	\label{EqFischerControl}
	\mathbf{f} &= (k_s + 1) \mathbf{e}_2 - (k_s + 1) \mathbf{e}_2(0) + \nu \\
	\mathbf{\dot{\nu}} &= (k_s + 1) \alpha_2 \mathbf{e}_2 + \beta \mathrm{sgn}(\mathbf{e}_2) \\
		\label{EqFischere2}
	\mathbf{e}_2 &= (\mathbf{\dot{p}}_d - \mathbf{\dot{p}}) + \alpha_1 (\mathbf{p}_d - \mathbf{p}) \\
	\label{EqPIDControl}
	\mathbf{u} &= k_p \mathbf{q}_d + k_i \int \mathbf{q}_d dt + k_d \mathbf{\dot{q}_d}
\end{align}

\noindent where $k_s >0$ and $\alpha_2 > 1/2$ are control gains for the translational controller and $k_p ,k_i,k_d > 0$ are the PID controller gains for the desired Euler angles, $\mathbf{q}_d$, where $\mathbf{q}_d$ are determined from $\mathbf{f}$.

\section{Simulation Results}
\label{SecSim}
To demonstrate the algorithm capabilities, we examine two simulation scenarios: (1) a single vehicle maneuvers around various moving obstacles to a goal position, and (2) multiple vehicles navigate into two entrances of a building to reach goal positions inside the building. We ran the simulations for multiple random initial conditions, and show an example result from each scenario. For all simulations we use Eq.~\ref{EqTauFTh} in Theorems 1 and Eqs.~\ref{EqThvc}-\ref{EqVcSensorTauF}, \ref{EqVcObsIneq} and \ref{EqVcObsTauF} in Theorem 2 as equality constraints. The other simulation parameters and results are discussed in the following sections. 

We introduce a smoothing method to reduce the vehicle switching between the conservative and aggressive tangent directions. This variable, $r_{CA}$, defined in Eq.~\ref{EqRCA}, is similar to $r_{c}^*$ from Eq.~\ref{EqRcStar} to account for the distance the vehicle travels when making a turn. This distance allows the vehicle to safely use the aggressive tangent direction until it is within $r_{CA}$ of the obstacle/vehicle. Once within $r_{CA}$, it has sufficient clearance to maneuver and reverts back to the criteria for conservative and aggressive tangent directions from Eq.~\ref{EqAggresConserv}.
\begin{align}
		\label{EqRCA}
		r_{CA} &= r_c + r_{180} + \Delta T_{tot} v_{\perp,CA} \\
		v_{\perp,CA} &= \mathbf{\dot{p}}_k \cdot \frac{\mathbf{p}_d - \mathbf{p}_{min}}{||\mathbf{p}_d - \mathbf{p}_{min}||}
	\end{align}
\noindent where $r_{180}$ and $\Delta T_{tot}$ are defined in Eq.~\ref{EqRcStar}.

\subsection{Simulation 1}
The first simulation shows a vehicle navigating around moving obstacles to a goal position. The environment has a bounded mean wind disturbance of 3 m/s and there is also gusting. The wind model uses the Von K\'arm\'an power spectral density function over a finite frequency range, then applies that model to the method described in \cite{Cole2013} to create a spatial wind field. The gusting profile is defined in the military specification MIL-F-8785C \cite{MilWind} as a ``$1-\cos$" model. Figure \ref{FigSimExWind}A shows the wind experienced by the vehicle in the simulation.

The vehicle parameters for the simulation are: $m = 0.54$g, $f_{max}=10.17$N, $r_c = 2$m, $r_{obs}=7$m, $r_s=10$m, $\Delta T_s = 1$s, $\Delta T_c = 0.1$s, $J = \mathrm{diag}([0.0017,0.0017,0.0031])$ kg/m$^2$, $C_d$ = 1.6, and $A_{x_w}= [0.20,~0.20,~0.45]$m$^{2}$. The maximum cruise velocity is solved from Theorem 2 as $v_c = 0.89$m/s. The controller gains are $\alpha_1 = 1$, $\alpha_2 = 3$, $k_s = 9$, and $\beta = 0.25$, where all of these values satisfy the constraints outlined in \cite{FischerNL2014} except $\beta$, which produces non-smooth behavior for large values. 

There are four moving obstacles in the simulation with velocity magnitudes ranging from 0.375 m/s to 0.75 m/s and constant courses as shown in Fig.~\ref{FigSim1Snapshots}A. Note that all obstacle velocities are less than the vehicle cruise velocity. Snapshots of the vehicle maneuvering through the environment are shown in Fig.~\ref{FigSim1Snapshots}, the course changes are shown in Fig.~\ref{FigSim1DelPhi}, and the thrust required is shown in Fig.~\ref{FigSim1Thrust}. The vehicle clears all four obstacles by $\geq r_c$, travels at $v_c$ until the goal position, and reaches the goal position in just under 90 seconds. 

\begin{figure}
	\centering
		\includegraphics[width=2.95in]{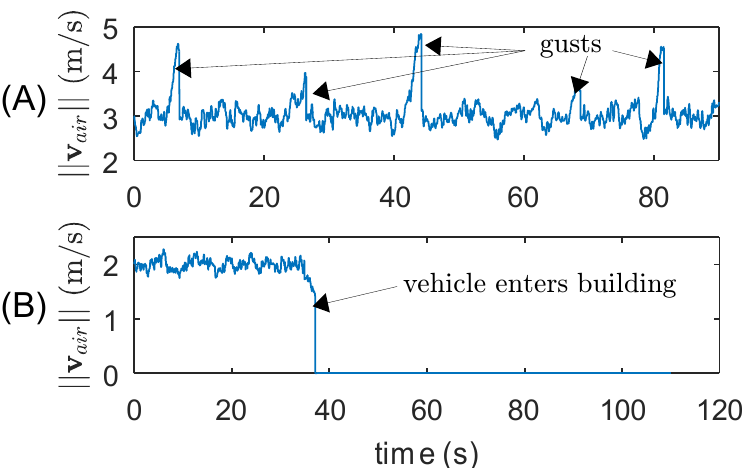}
		\caption{\label{FigSimExWind} (A) Wind experienced by the vehicle in the first simulation with some gusting throughout the simulation. (B) Wind experienced by vehicle 5 in the second simulation. The wind speed reduces to 0 as the vehicle enters the building. }
\end{figure}

\begin{figure*}
	\centering
	\includegraphics[width=0.9\textwidth]{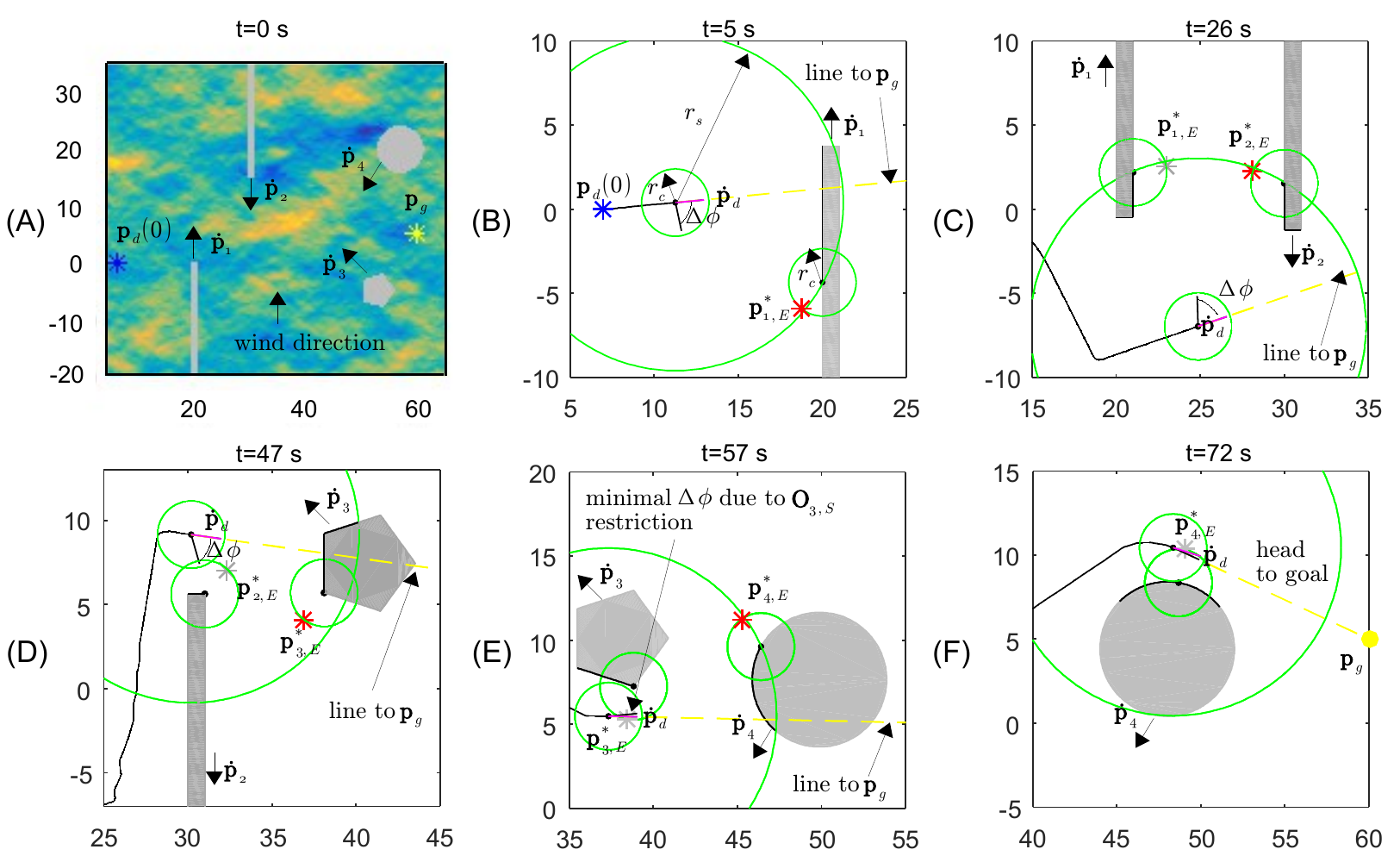}
	\caption{\label{FigSim1Snapshots} (A) Simulation environment with four moving obstacles between the vehicle's starting position and desired goal position. The windfield is shown at $t=0$. (B) The vehicle identifies the first obstacle, chooses the extent to navigate towards and the corresponding projected extent point $\mathbf{p}_{1,E}^*$. The vehicle uses the conservative tangent direction to determine the desired $\Delta \phi$. (C) The vehicle has cleared the first obstacle, heads toward the goal position briefly, and now identifies the course change to navigate around the second obstacle. The vehicle can realize the desired course change for the second obstacle since $\Delta \phi_{2} \in \mathbf{O}_{1}$. Again, the vehicle uses the conservative tangent direction and goes behind the obstacle. (D) The vehicle has traversed the second obstacle sufficiently, briefly heads to the goal position, and is now starting to navigate around the third obstacle. The vehicle is again choosing the conservative tangent direction, but the second obstacle restricts the course change because $\Delta \phi_{3} \notin \mathbf{O}_{2}$. (E) The vehicle has sufficiently cleared the third obstacle so that $\Delta \phi_g \in \mathbf{O}_{3}$; however, $\Delta \phi_g \notin \mathbf{O}_{4}$, so it must navigate around obstacle four. The desired course change to navigate around obstacle four, $\Delta \phi_4 \notin \mathbf{O}_3$, so the vehicle gets as close to $\Delta \phi_4$ as it can from the feasible angles in $\mathbf{O}_3$. (F) The vehicle has sufficiently traversed the fourth obstacle such that $\Delta \phi_g \in O_{4}$, so the vehicle heads to the goal. }
	
\end{figure*}

\begin{figure}
	\centering
		\includegraphics[width=2.9in]{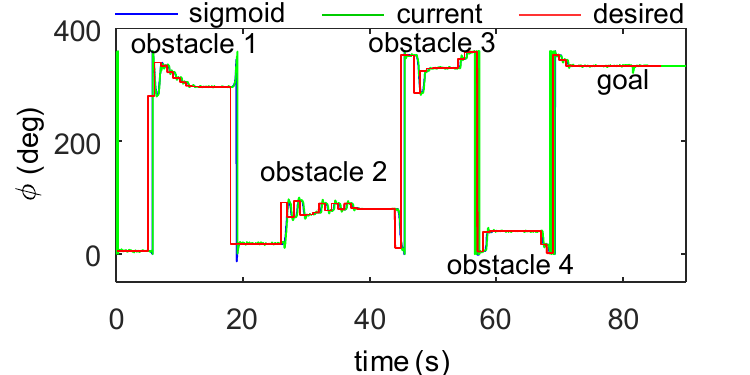}
		\caption{\label{FigSim1DelPhi} Course changes as the vehicle navigates in the environment. The fluctuations in course change as the vehicle navigates each obstacle is due to switching between the conservative and aggressive tangent directions as the vehicle gets new sensor information.} 
	
\end{figure}

\begin{figure}
	\centering
		\includegraphics[width=2.9in]{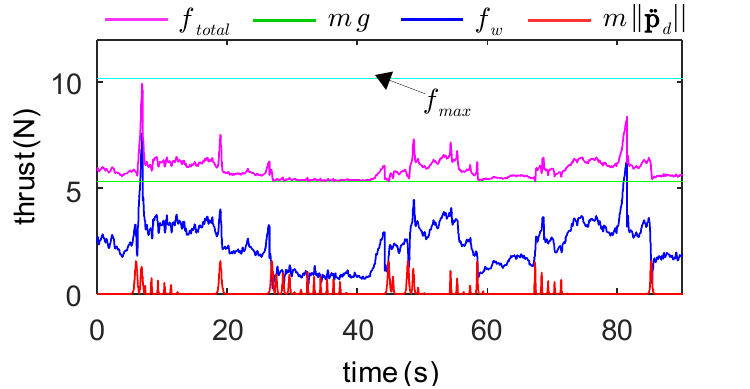}
		\caption{\label{FigSim1Thrust} Thrust required to navigate in the environment. The total thrust required is less than the maximum thrust available.} 
\end{figure}

\subsection{Simulation 2}
The second simulation shows five vehicles navigating into a building and around stationary obstacles to different goal positions. We use temporary goal positions to guide the vehicles inside the building, then once the temporary goal positions are reached the final goal positions are used. There is a bounded mean disturbance of 2 m/s when the vehicles are outside the building, a small transition zone where the wind enters the building, and no wind once the vehicles are fully inside. The transition zone is based on the results of a simulation of the building environment using SolidWorks 2016 Flow Simulation package. Figure \ref{FigSimExWind}B shows the wind experienced by vehicle 5 in the simulation.

The vehicle parameters that differ for the five vehicles are summarized in Table \ref{TblSimParams}. The vehicles are physically the same but differ in maximum thrust and clearance radii to represent vehicles carrying different payloads for different missions. The other parameters are the same as in simulation 1, except $r_{obs} = 2$m. 

\begin{table}
	\centering
		\caption{\label{TblSimParams} Simulation Parameters}
		\begin{tabular}{c|cccccc}
		Vehicle &  1 & 2 & 3 & 4 & 5 \\ \hline
		$f_{max}$ (N) & 10.17 & 10.73 & 9.6 & 9.1 & 10.17\\ \hline
		$r_c$ (m)  & 0.65 & 0.55 & 0.40 & 0.60 & 0.50 \\ \hline
		$v_c$ (m/s) & 0.29 & 0.38 & 0.51 & 0.34 & 0.42 \\ \hline
		\end{tabular}
\end{table} 

Different cruise velocities result in different sets $\mathcal{I}_{mnvr}$ for each vehicle. In this simulation, vehicle 1 does not maneuver around any other vehicles and vehicle 3 must maneuver around all the other vehicles. Figure \ref{FigSim2Overview} shows an overview of the vehicle trajectories overlaid on the windfield at one time instance. We examine the performance of vehicle 5 as a representative case in Figs.~\ref{FigSim2DelPhiDelV} to \ref{FigSim2Snapshots}, showing snapshots of the vehicle navigating the environment, the course and velocity changes, and thrust required, respectively. All the vehicles clear all of the other obstacles/vehicles by their desired $r_c$ values, and the maximum thrust is not violated for any vehicle. The computation time per sensor update was 0.15 seconds or less (for $\leq 600$ sensed points) which is well below the update rate of 2 samples/second.

\begin{figure}
	\centering
		\includegraphics[width=2.9in]{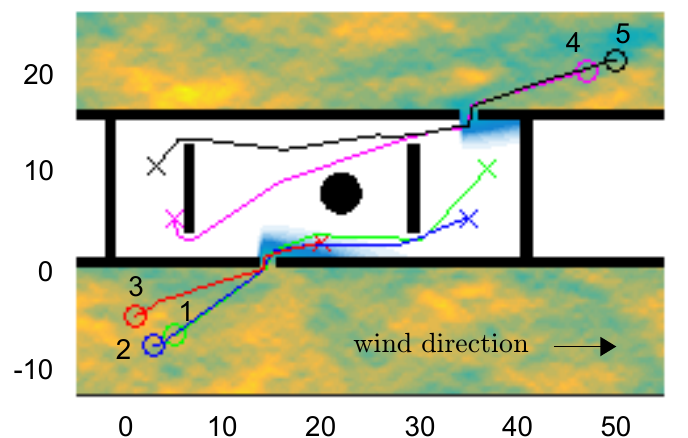}
		\caption{\label{FigSim2Overview} Overview of all the vehicle trajectories navigating into the building overlaid on the windfield at one time instance.}
	
\end{figure}

\begin{figure}
	\centering
		\includegraphics[width=2.9in]{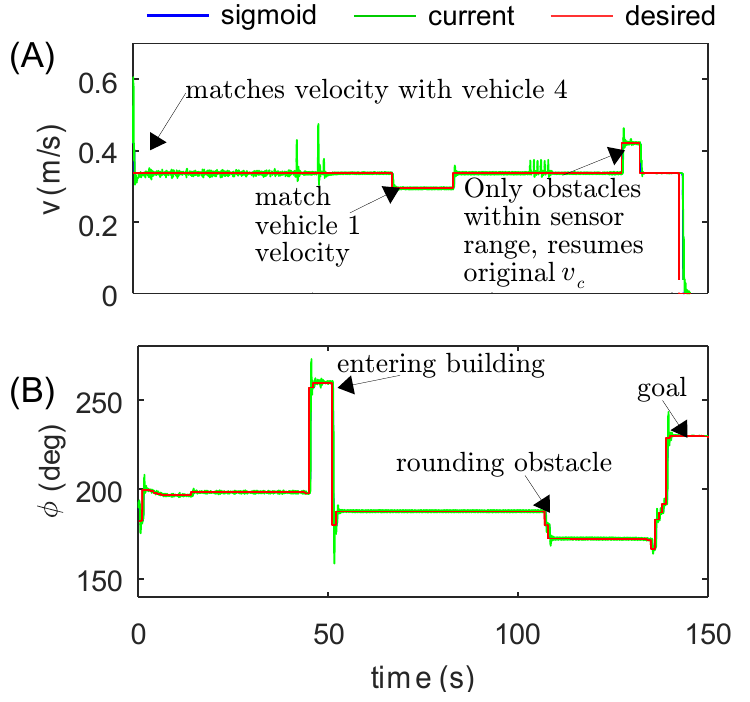}
		\caption{\label{FigSim2DelPhiDelV} (A) The vehicle immediately matches the velocity of vehicle 4 because there are both obstacles and vehicles within sensor range. As the vehicles enter the building, vehicle 5 slows down further because it has come within $r_{c,4}^*$ of vehicle 4. Once vehicle 5 is outside of $r_{c,4}^*$ it resumes its previous velocity. As the vehicle's continue to navigate, vehicle 5 increases velocity to its cruise velocity when there are only obstacles within sensor/communication range. Just before the goal position vehicle 5 is within communication range of vehicle 4 so it matches the velocity of vehicle 4. (B) Course changes for vehicle 5 as it navigates in the environment.}
	
\end{figure}

\begin{figure}
	\centering
		\includegraphics[width=2.75in]{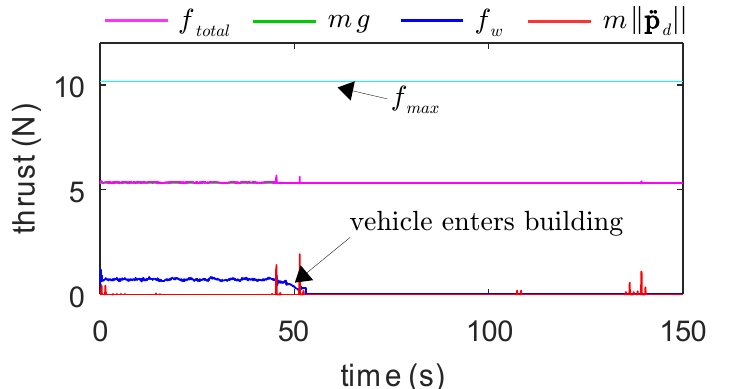}
		\caption{\label{FigSim2Thrust} Thrust required for vehicle 5. The required thrust is well below the maximum thrust since the vehicle is at a reduced velocity for almost all of the simulation and also does not experience any wind once inside the building.}
\end{figure}

\begin{figure*}
	\centering
	\includegraphics[width=0.9\textwidth]{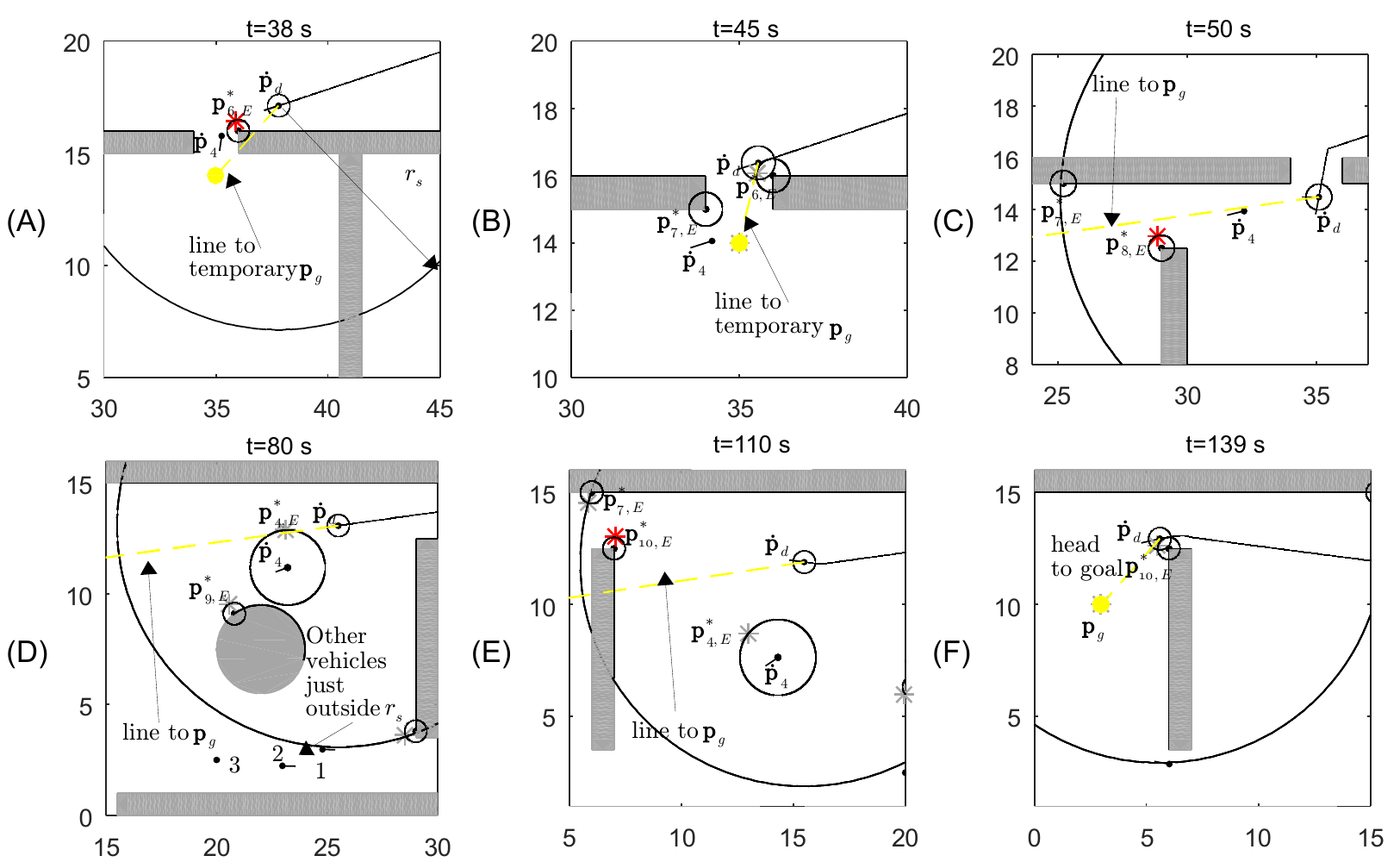}
	\caption{\label{FigSim2Snapshots} Snapshots as vehicle 5 traverses into the building to the goal position. (A) The vehicle has matched the velocity of vehicle 4, and has identified two obstacles within sensor range as it heads towards the temporary goal position. For this example we reserve $k=1$ to $k=5$ for the other vehicles and start identifying the obstacles at $k=6$. Since $\Delta \phi_g \notin \mathbf{O}_{6}$, the vehicle makes a course change of $\Delta \phi_6$. (B) The vehicle has cleared obstacle 6 sufficiently that $\Delta \phi_g \in \mathbf{O}_6$. Since $\Delta \phi_g \in \mathbf{O}_7 \cap \mathbf{O}_4$, the vehicle makes a course change to the temporary goal position. (C) The vehicle has reached the temporary goal and since $\Delta \phi_g \in \mathbf{O}_7 \cap \mathbf{O}_8 \cap \mathbf{O}_4$ the vehicle makes a course change of $\Delta \phi_g$. (D) None of the obstacles or vehicle restrict vehicle 5 from heading towards the goal position. Even though none of other vehicles enter sensor range, vehicle 1 with $||\mathbf{\dot{p}}_1|| < ||\mathbf{\dot{p}}_4||$ is within sensor range of vehicle 4, so both vehicles 4 and 5 match $||\mathbf{\dot{p}}_1||$. (E) The vehicle make a course change to start navigating around obstacle 10 since it is the only obstacle or vehicle that does not allow maneuvering towards the goal position. (F) The vehicle has sufficiently traversed obstacle 10 that it can head to the goal position. Additionally, since there are no other vehicles within sensor range, the vehicle can resume its cruise velocity. A video of the vehicle navigating in the building is available with the supplemental online material.}
	
\end{figure*}

\section{Conclusion}
\label{SecConclusion}
The trajectory generator presented navigates a vehicle in an unknown environment collision-free while respecting the vehicle's physical limitations. The vehicle uses its sensor and communication inputs to compute course and velocity changes to avoid obstacles by a prescribed clearance distance. The sigmoid functions used to transition course and velocity provide piecewise smooth motion with bounded discontinuities and incorporate the course changes from each sensor update by matching the sigmoid slopes and summing the curves. In the event the feasible course changes become an empty set, the vehicle adjusts priority temporarily to avoid collisions. Lastly, the vehicle incorporates the expected wind disturbance, thrust limitations, and sensor constraints to bound the maximum safe cruise velocity. 

There are several directions in which the algorithm presented in this article can be extended. One area is obstacles with non-constant velocity and course. If the algorithm considered everything in the environment as if it were another vehicle, it would handle maneuvering obstacles as well. Using this approach, the algorithm would generate more conservative trajectories for obstacles with constant velocity and course than the trajectories presented in this article. Moreover, the algorithm is defined generically enough that it could be extended to 3D motions by rotating the plane in which the motion occurs or generating a separate altitude adjustment that is combined with the planar trajectory. The thrust required for the altitude maneuver reduces the thrust available for planar motions, so when the two motions are combined the vehicle's thrust limitation is still respected. Finally, even though the simulations presented use stationary goal positions, the algorithm can handle moving goal positions so it can be incorporated into a higher-level motion planner.

\section*{Appendix A: Proof of Theorem \ref{ThTauF}}
The timespan of the $n^{th}$ sigmoid curve, $\tau_{f,n}$, must be set appropriately to ensure that the vehicle does not violate its maximum thrust when performing course and velocity changes. We examine the maximum acceleration of the $n^{th}$ sigmoid curve in conjunction with any sigmoid curves it may be summed with to show that the value of $\tau_{f,n}$ for any sigmoid curve does not cause the vehicle's maximum thrust constraint to be exceeded. Additionally, we develop a constraint on the maximum wind speed related to the vehicle's thrust and drag properties.

\begin{proof}
The maximum acceleration of the summed sigmoid curves, $\phi$ and $v$, from Eqs.~\ref{EqSigPhiSum} and \ref{EqSigVSum}, respectively, at any time $t$ is given by Eq.~\ref{EqAccelMaxMag} and repeated here:
\begin{equation}
	||\mathbf{\ddot{p}}_d|| = \sqrt{v^2 \dot{\phi}^2 + \dot{v}^2}
\end{equation}
The total thrust is given by Eq.~\ref{EqFplanarNormalDrag}, where the thrust is maximized when the trajectory acceleration, $\mathbf{\ddot{p}}_d$, is aligned with the drag, $\mathbf{f}_w$. Since this is the constraining case, we substitute Eq.~\ref{EqAccelMaxMag} into Eq.~\ref{EqFplanarNormalDrag} and re-write Eq.~\ref{EqFplanarNormalDrag} in terms of the vector magnitudes as
\begin{equation}
	\label{EqFplanarMaxTauFIneq}
	f_{p,max} \geq m \sqrt{v^2 \dot{\phi}^2 + \dot{v}^2} + K_d v_{w,max}^2 
\end{equation}
\noindent where $v_{w,max} \leq v_c + v_{air,max}$. This inequality holds for any time $t$.

We conservatively bound the maximum thrust by maximizing each term in the right hand side of Eq.~\ref{EqFplanarMaxTauFIneq} independently. We start with the maximum acceleration from the trajectory, which occurs where $\frac{d||\mathbf{\ddot{p}}_d||}{dt} = 0$ and gives

\begin{align}
	 \label{EqMaxDdotPdSig} v^2 \dot{\phi} \ddot{\phi} +  \dot{\phi}^2 v \dot{v} + \dot{v} \ddot{v} &= 0 
\end{align}

We use Eqs.~\ref{EqFplanarMaxTauFIneq} and \ref{EqMaxDdotPdSig} in an inductive argument to show that the sum of any $n$ sigmoids does not violate $a_{max}$. To do this, we simplify Eq.~\ref{EqMaxDdotPdSig} by utilizing the terms in Sec.~\ref{SecSigmoid}. We evaluate Eq.~\ref{EqMaxDdotPdSig} at $\tau_{n,max}$ by substituting $H_n$ from Eq.~\ref{EqHDef} and the sigmoid function definitions from Eqs.~\ref{EqSigPhi} to \ref{EqSigVdot} into Eq.~\ref{EqMaxDdotPdSig} and simplifying as follows:
\begin{align}
	v^2 \dot{\phi} \ddot{\phi} +  \dot{\phi}^2 v \dot{v} + \dot{v} \ddot{v} = 0& \nonumber \\
	(d_1H_n+d_4)^2\left(c_1c_2(1-H_n^2)\right)\left(-2c_1c_2^2H_n(1-H_n^2)\right) + & \nonumber \\  \left(c_1c_2(1-H_n^2)\right)^2(d_1H_n + d_4))\left(d_1 d_2(1-H_n^2)\right) + & \nonumber \\
	\left(d_1 d_2(1-H_n^2)\right)\left(-2 d_1 d_2^2 H_n (1-H_n^2)\right) = 0&  \nonumber \\
	\label{EqDerivMax0} -3c_1^2 d_1^2 H_n^3 - 5 c_1^2d_1 d_4 H_n^2 + &\nonumber \\ (-2 c_1^2 d_4^2 + d_1^2 c_1^2 - 2 d_1^2) H_n + d_1 d_4 c_1^2 = 0&
\end{align}
\noindent where $d_4$ is modified from the sigmoid coefficient definition in Eq.~\ref{EqCoeffSummaryEnd} for this proof to include the velocity $v_{n-1}$, which is the desired final velocity at $t_{n-1}$ (i.e.~from the previous sigmoid function): 
\begin{eqnarray}
	\label{EqAlternateD4}
	d_4 &=& v_{n-1} + \frac{1}{2} \Delta v_n 
\end{eqnarray}
The final result in Eq.~\ref{EqDerivMax0} is a cubic polynomial in $H_n$. Since all the coefficients are known, the roots are solvable. Recall that $H_n = \tanh(c_2 \tau_{n,max} - c_3)$ (from Eq.~\ref{EqHDef}); therefore, to be a physically meaningful solution, the roots must be real and satisfy $|H_n| < 1-\varepsilon$. For any value of its coefficients, Eq.~\ref{EqDerivMax0} has one real root that satisfies $|H_n| < 1-\varepsilon$ due to the unimodal shape of the acceleration curve. This result is proven in Appendix D.

The solution to Eq.~\ref{EqDerivMax0}, is used in Eq.~\ref{EqKDef} to give the ratio of the time of maximum acceleration to the total curve span, thus indicating the shape of the acceleration curve. To bound the timespan, we introduce the constraint on acceleration in Eq.~\ref{EqTauF1Proof}, where the solution to $H_n$ from Eq.~\ref{EqDerivMax0} is used to evaluate the $S_{traj}$ term. All the terms needed to bound the first term in Eq.~\ref{EqFplanarMaxTauFIneq} are consequently defined.

The second term in Eq.~\ref{EqFplanarMaxTauFIneq} for the drag force includes the known variable $K_d$, and we define $v_{w,max} =  \mathrm{max}(v_{n-1}, v_{n-1} + \Delta v_n) + v_{air,max}$, for the $n^{th}$ sigmoid. We have now defined all the terms in Eq.~\ref{EqFplanarMaxTauFIneq} so we can bound $\tau_{f,n}$.

We start by re-arranging Eq.~\ref{EqFplanarMaxTauFIneq}, substituting Eqs.~\ref{EqAccelMaxMag} and \ref{EqAccelMaxMagSigVars} for the sigmoid acceleration, and Eq.~\ref{EqAmaxDef} for the known acceleration terms, and simplifying to give 
\begin{align}
	\frac{1}{m} \left(f_{p,max} - K_d v_{w,max}^2\right) &\geq ||\mathbf{\ddot{p}}_{d,max}|| \nonumber \\
	\frac{1}{m}\left(f_{p,max} - K_d v_{w,max}^2\right) &\geq \frac{2 c_3}{\tau_f}\sqrt{S_{traj}}  \nonumber \\
	a_{max} &\geq \frac{2 c_3}{\tau_f} \sqrt{S_{traj}}  \nonumber \\
	\label{EqTauF1Proof}
	\tau_{f,min,n} &\geq \frac{2 c_3}{a_{max}}\sqrt{S_{traj}}
\end{align}
\noindent where $S_{traj}$ (Eq.~\ref{EqStrajDef}) is also known since $H_n$ is the solution to Eq.~\ref{EqDerivMax0}.

Equation \ref{EqTauF1Proof} is the minimum sigmoid curve timespan that does not violate $f_{p,max}$. This equation is utilized to maximize the planar thrust independent of any previous sigmoid functions, and we use it in two cases. The first case is when the longest of the previous sigmoid curves is complete and satisfies
\begin{equation}
 	t_n \geq \max\limits_{i<n}\left(t_{o,i} + \tau_{f,i}\right)
\end{equation}
\noindent which means that the sigmoid curve is starting after all other previous sigmoid curves have completed. 

The second case is when $\Delta \phi_n$ and $\Delta v_n$ are significant compared to $\Delta \phi_{n-1}$ and $\Delta v_{n-1}$, so that the slope of the $n^{th}$ curve satisfies $h_n^+ \leq h_{n-1}^-$. If we matched slopes for this case, but $K_n \tau_{f,n} > (1-K_{n-1}) \tau_{f,n-1}$, then the maximum acceleration is violated as follows:
\begin{equation}
	\label{EqAmaxSlopeTooBig}
	a_n = \frac{a_{max}}{(1-K_{n-1}) \tau_{f,n-1}} K_n \tau_{f,n} > a_{max}
\end{equation}
To determine if this is the case, we use $\Delta \phi_n$ and $\Delta v_n$ to solve for $\tau_{f,min,n}$ (Eq.~\ref{EqTauFMinDef}) and $K_n$ (Eq.~\ref{EqKDef}) to compare to $\tau_{f,n-1}$ and $K_{n-1}$ as follows: 
\begin{align}
	h_n^+ &\leq h_{n-1}^-  \nonumber \\
	\frac{a_{max}}{K_n \tau_{f,min,n}} &\leq \frac{a_{max}}{(1-K_{n-1}) \tau_{f,n-1}}  \nonumber \\
	\label{EqSlopeCheck}
	\frac{(1-K_{n-1})\tau_{f,n-1}}{(1-K_n)\tau_{min,n}} &\leq 1
\end{align}
If Eq.~\ref{EqSlopeCheck} is satisfied, then Eq.~\ref{EqTauF1Proof} is the solution for the sigmoid curve timespan that respects $a_{max}$. If the inequality is violated, then the slope of the next sigmoid must be matched to the previous one to determine the value of $\tau_{f,n}$ so that the summation does not exceed $a_{max}$. 

The maximum acceleration is a function of $h_{n-1}$, $\tau_{f,n}$, and $K_n$ in Eqs.~\ref{EqTauFMinDef} and \ref{EqSlopePlus}. Since the maximum acceleration appears in the denominator of Eq.~\ref{EqSlopeMinus}, we define $K_{min} = \mathrm{min}(1-K_n,K_n)$ so that $\tau_f$ is as large (and conservative) as possible. Combining these two cases, the maximum acceleration of the $n^{th}$ sigmoid is given by
\begin{align}
\label{EqAmax2}
	a_{h_n} =  \left\{\begin{array}{ll}
	a_{max}, & \frac{(1-K_{n-1})\tau_{f,n-1}}{K_{n}\tau_{min,n}} \leq 1 \\
	 \tau_{f,n}K_{min}h_{n-1}^-, & \mathrm{otherwise}
	\end{array}
	\right.
\end{align}
\normalfont
\noindent Substituting Eq.~\ref{EqAmax2} into Eq.~\ref{EqTauF1Proof} and simplifying produces the following:
\begin{align}
	\tau_{f,n} &\geq \frac{2 c_3}{a_{h_n}}\sqrt{S_{traj}}  \nonumber \\
	\tau_{f,n}^2 &\geq \frac{2 c_3}{K_{min}h_{n-1}^-}\sqrt{S_{traj}} \nonumber \\
	\label{EqTauF2Proof}
	\tau_{f,n} &\geq \sqrt{\frac{2 c_3}{K_{min}h_{n-1}^-}\sqrt{S_{traj}}}
\end{align}
Equation \ref{EqTauF2Proof} gives the value of $\tau_{f,n}$ when the slopes of the sigmoid curves must be matched. All of the cases described are summarized in Eq.~\ref{EqTauFTh}. In the simulation we use Eq.~\ref{EqTauFTh} as an equality constraint. The rest of this section shows that we do not violate the constraint on maximum thrust if Eq.~\ref{EqTauF2Proof} is satisfied.

The use of the slope matching respects $a_{max}$ and $f_{p,max}$ because the linear slope estimation overestimates the sum of two successive sigmoid acceleration curves when the offset time, $t_{o,n}$, is defined by Eq.~\ref{EqTimeOffset}. The more constraining case is when $t_{o,n}$ is dependent on $t_{int,n}$. Since $t_{int,n}$ is the intersection point between the linear approximation and sigmoid curve acceleration it satisfies Eqs.~\ref{EqTintIneq} and \ref{EqTintDef}.
Both curves are monotonically decreasing after the maximum acceleration point; therefore, there is only one intersection point and one solution to Eqs.~\ref{EqTintIneq} and \ref{EqTintDef}, and the linear approximation is larger than the sigmoid curve acceleration after this point. Thus, if the sum of the linear approximations does not violate $a_{max}$, then the actual acceleration does not either. We can express this as
\begin{align}
\label{EqThALessThanAmax}
	\sum_{i=2}^{n} ||\mathbf{\ddot{p}}_{d,i-1} (t-t_{o,i-1})|| \leq 
	\sum_{i=2}^n \left(a_{lin,i-1}(t-t_{o,i-1}) + a_{lin,i} (t-t_{o,i}) \right) \leq a_{max}
\end{align}
\noindent for $n \geq 2, ~\forall t \geq t_{int,1}$, where 
\begin{equation}
	a_{lin,n} = \left \{\begin{array}{ll}
	\frac{-a_{max}}{(1-K_n) \tau_{f,n}} t_{lin} + \frac{a_{max}}{1-K_n}, & K_n \tau_{f,n} \leq t_{lin} \leq \tau_{f,n} \\
	\frac{a_{max}}{K_n \tau_{f,n}} t_{lin}, & 0 \leq t_{lin} \leq K_n \tau_{f,n} \\
	0,& \mathrm{otherwise}
	\end{array}
	\right.
\end{equation} 
In summary, Eq.~\ref{EqTauF1Proof} provides a valid timespan solution for the $n=1$ case, and Eq.~\ref{EqTauF2Proof} provides a solution for any $n$ assuming the previous trajectories do not violate $a_{max}$. Therefore, by induction, this procedure provides a feasible solution for any number of summed sigmoids by induction.

The final condition of Theorem \ref{ThTauF} is the restriction on $v_{air,max}$. If we consider that the trajectory force vector from Eq.~\ref{EqFplanarNormalDrag} is a normal force, $\mathbf{f}_n$, with maximum magnitude defined by
\begin{equation}
	\label{EqFnMaxDef}
	f_{n,max} = m \frac{v_c^2}{r_{min}}
\end{equation}
\noindent where $r_{min}$ is user defined, then Eq.~\ref{EqFplanarMaxTauFIneq} is re-written as
\begin{equation}
	\label{EqFplanarMaxTauFIneqNormalForce}
	f_{p,max} \geq m \frac{v_c^2}{r_{min}} + K_d v_{w,max}^2 
\end{equation}
\noindent Additionally, for this analysis we re-define $v_{w,max} = v_c + v_{air,max}$ so that we can write Eq.~\ref{EqFplanarMaxTauFIneqNormalForce} as a quadratic in $v_c$ as follows:
\begin{align}
a_{vc} &v_c^2 + b_{vc} v_c + c_{vc} =0 \\
	a_{vc} &= \frac{m}{r_{min}} + K_d \\
	b_{vc} &= 2 K_d v_{air,max} \\
	\label{EqC_vcDef}
	c_{vc} &= K_d v_{air,max}^2 - f_{planar}
\end{align}
\noindent The roots are then  
\begin{equation}
	v_c = \frac{ -b_{vc} \pm \sqrt{b_{vc}^2 - 4 a_{vc} c_{vc}}}{2 a_{vc}} 
\end{equation}
\noindent To be a physically meaningful solution for $v_c$, the roots must be real and positive, which means $b_{vc}^2 - 4 a_{vc} c_{vc} \geq 0$ and $-b_{vc} + \sqrt{b_{vc}^2 - 4 a_{vc} c_{vc}} > 0$. Re-arranging these two inequalities gives:
\begin{align}
	4a_{vc} c_{vc} &\leq b_{vc}^2 \\
	\label{EqLimitingCase}
	4a_{vc} c_{vc} &< 0
\end{align}
\noindent which shows that the second inequality is the more restrictive constraint. Since $a_{vc} > 0$, Eq.~\ref{EqLimitingCase} reduces to $c_{vc} \leq 0$. Substituting in Eq.~\ref{EqC_vcDef} gives
\begin{align}
	c_{vc} &\leq 0 \\
	K_d v_{air,max}^2 - f_{planar} &\leq 0 \\
	\label{EqVairMax}
	v_{air,max}& \leq \sqrt{\frac{f_{planar}}{K_d}}
\end{align}
\noindent where Eq.~\ref{EqVairMax} provides the maximum wind speed in which the vehicle can safely fly. 

\end{proof}

\section*{Appendix B: Proof of Theorem 2}
There are three inequality constraints in Theorem 2 as stated in Eq.~\ref{EqThvc}. Each inequality bounds the cruise velocity due to a different parameter. The first is due to the vehicle's thrust, the second is due to the sensor range and update rate, and the third is due to the obstacle spacing. The minimum of the resulting bounds thus satisfies all three. We examine each inequality separately in this proof. 

\subsection{Thrust Constraint, $v_{c,f}$}
\begin{proof}
The first constraint on $v_c$ is due to the vehicle thrust limitations. Equation \ref{EqFplanarMaxTauFIneqNormalForce} is re-written as Eq.~\ref{EqFplanarInequalityTh}, which provides the inequality constraint for $v_c$ to respect the vehicle's maximum thrust capability and turn radius requirement. 
\end{proof}

\subsection{Sensor Range and Rate Constraint, $v_{c,s}$}
\begin{proof}
The second constraint on $v_c$ is due to sensor limitations; namely, the vehicle must be able to react to obstacles as they are detected to avoid them. We consider a vehicle traveling towards an obstacle, where the velocity vector of the vehicle is opposite the velocity vector of the obstacle, as shown in Fig.~\ref{FigSensorRangeIncremental}. In the worst case scenario, the vehicle cannot react immediately due to time delays from either not sensing the obstacle immediately or being in the middle of a maneuver and not having any available thrust. We bound the cruise velocity with the most constraining case.

We define the maximum time delay from Eq.~\ref{EqTdmax} and repeat it here for clarity:
\begin{equation}
	t_{d,max} = \left\{\begin{array}{ll}
	\frac{0.54c_3\pi}{a_{max}}v_{c,s} + \Delta T_c, & 0.54 \tau_{f,s} \geq 2\Delta T_s\\
	2 \Delta T_s + \Delta T_c, & \mathrm{otherwise}
	\end{array}
	\right.
\end{equation}
\noindent where the first condition is the delay resulting from thrust availability (which we discuss next) and the second condition is the delay resulting from not sensing the obstacle immediately then waiting an additional sensor update to estimate velocity. Both conditions include the maximum computation time to process the sensor data into course and velocity changes.

We define the first condition as the worst case delay due to the vehicle being mid-maneuver from the previous course and velocity change. Recall from Sec.~\ref{SecSigmoid} that the earliest the next maneuver, in this case the maneuver shown in Fig.~\ref{FigSensorRangeIncremental}, can start is when there is thrust available, which starts when $t \geq t_{o,n-1} + t_{int,n}$. Without loss of generality, we assume that the previous maneuver is the first maneuver, in which case $t_{o,n-1} = \Delta T_c$. Next, we need a solution for $t_{int,n}$, which is dependent on the previous maneuver. 

For the previous maneuver we assume a maximum timespan from a worst case course change, $\Delta \phi_{n-1} = \pi$. Since $\Delta \phi_{n-1} > \pi/2$, we satisfy the condition to use the constant velocity case as the most constraining case, as established in Appendix E. Additionally, since we assume it is the first maneuver, the sigmoid curve timespan is defined by Eq.~\ref{EqTauFMinDef} and simplified as follows:
\begin{align}
	\tau_{f,n-1} &= \tau_{f,min,n-1} = \frac{2 c_3}{a_{max}} \sqrt{S_{traj}} \nonumber \\
	&= \frac{2 c_3}{a_{max}} \sqrt{c_1^2 d_4^2} \nonumber \\
	&= \frac{2 c_3}{a_{max}} \frac{1}{2} \Delta \phi v_c^2 \nonumber \\
	\label{EqTauFN1ConstantVelGen}
	&=\frac{c_3 \Delta \phi}{a_{max}} v_c \\
	\label{EqTauFN1ConstantVel}
	&=\frac{c_3 \pi}{a_{max}} v_c
\end{align}
\noindent where $d_1 = 0$, $H = 0$, and $d_4 = v_c$ for the constant velocity case to simplify $S_{traj}$. 

For the constant velocity case $K_{n-1} =0.5$, so to solve for $t_{int,n}$, we simplify Eq.~\ref{EqTintDef} as
\begin{align}
\label{EqTintSimplify}
	\frac{-a_{max}}{0.5 \tau_{f,n-1}} t_{int,n} + 2 a_{max} =v_c \left(\frac{\pi c_3}{\tau_{f,n-1}} \left(1-\tanh^2\left(\frac{2c_3}{\tau_{f,n-1}}t_{int,n} - c_3\right)\right)\right)
\end{align}
If we define $t_{int,n} = K_{t_{int}} \tau_{f,n-1}$ and substitute this into Eq.~\ref{EqTintSimplify}, the resulting equation is only a function of $K_{t_{int}}$ since we know all the other terms ($c_3 = 3.8$ from Eq.~\ref{EqCoeffSummary1}, and $a_{max}$ is dependent on the vehicle):
\begin{align}
	\label{EqTintConstantVelSimplify}
	&\sqrt{2K_{t_{int}}-1} = \tanh\left(\left(2K_{t_{int}}-1\right)c_3\right) 
\end{align}
\noindent where Eq.~\ref{EqTintConstantVelSimplify} has no analytical solution but is solvable numerically to give $K_{t_{int}} = 0.5365$; therefore we use $t_{int,n} = 0.54 \tau_{f,n-1}$. 

Substituting the definition for $t_{int,n}$ into Eq.~\ref{EqTimeOffset} for the offset time, using $t_{o,n-1} = \Delta T_c$, and using the definition for $\tau_{f,n-1}$ from Eq.~\ref{EqTauFN1ConstantVel}, we get
\begin{equation}
	t_{o,n-1} = 0.54\tau_{f,n-1} + \Delta T_c = \frac{0.54c_3 \pi}{a_{max}} v_{c,s} + \Delta T_s
\end{equation} 
\noindent where we use $v_{c,s}$ as the cruise velocity. This is the same as the first condition in Eq.~\ref{EqTdmax}.

\begin{figure}
	\centering
		\includegraphics[width=2.95in]{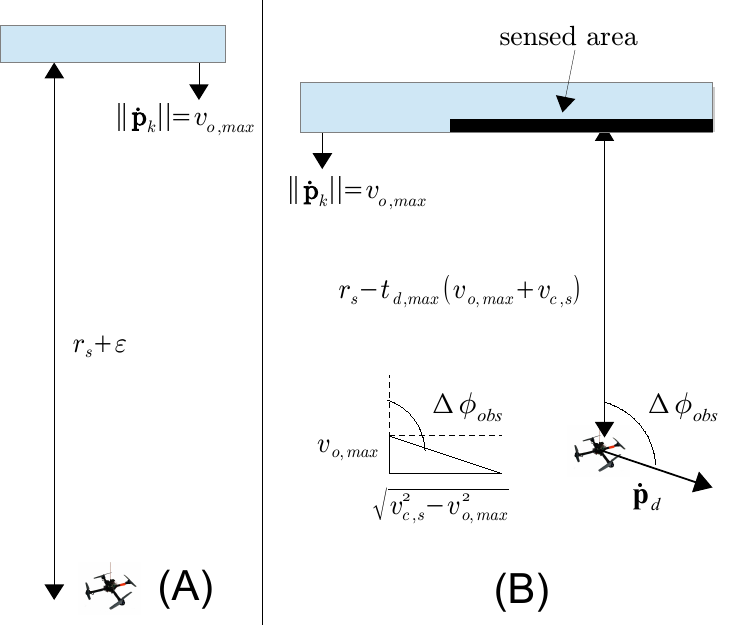}
		\caption{\label{FigSensorRangeIncremental} (A) The obstacle is just outside the vehicle's sensor range so it is not detected. (B) The vehicle detects the obstacle after a sensor update, then waits another sensor update to determine a velocity measurement and calculate a course change, $\Delta \phi_{obs}$. The vehicle must be able to make the course change and still clear the obstacle by $r_c$. } 
\end{figure}

Now that we have the maximum time delay, we compute the expected course change. For both cases of Eq.~\ref{EqTdmax}, the course change is calculated by Eq.~\ref{EqDelPhiObsKnown} for a known (or expected) maximum obstacle velocity in the environment of $v_{o,max}$ as shown in Fig.~\ref{FigSensorRangeIncremental}B: 
\begin{equation}
	\label{EqDelPhiObsKnown}
	\Delta \phi_{obs} = \tan^{-1}\left(\frac{v_{o,max}}{\sqrt{v_{c,s}^2 - v_{o,max}^2}}\right) + \frac{\pi}{2}
\end{equation}
\noindent This equation simplifies to $\Delta \phi_{obs} = \pi/2$, if the obstacle is stationary, and $\Delta \phi_{obs} \to \pi$ as the obstacle velocity approaches the vehicle cruise velocity. 

If there is information about $v_{o,max}$ it should be utilized, as it is undesirable to needlessly over-constrain the cruise velocity bound from the sensor, $v_{c,s}$; otherwise, the worst case where $v_{o,max} = v_{c,s}$ is assumed. In this latter case, the course change is worst when $\Delta \phi_{obs} = \pi$ because the vehicle must turn, match obstacle velocity, and clear the obstacle by $r_c$. The distance traveled is
\begin{equation}
	\label{EqDistanceSingleAxis}
	r_{180} = v_{c,s} \int_0^{\tau_{f,s}} \sin \phi(t) dt
\end{equation}
\noindent where $v(t)$ is the sigmoid function velocity, in this case it is constant, $\phi(t)$ is the sigmoid function course for a course change of $\Delta \phi_{obs}$, and $\tau_{f,s}$ is the sigmoid curve timespan.

To ensure that the cruise velocity is set appropriately and the vehicle does not outrun its sensor range, the following inequality must be satisfied:
\begin{align}
\label{EqIneq2Proof}
	v_{c,s} t_{d,max} &+ v_{c,s} \int_0^{\tau_{f,s}} \sin \phi(t) dt \leq r_s - \left(t_{d,max} + \tau_{f,s}\right)v_{o,max} - r_c
\end{align}
\noindent which takes into account the delay in starting the maneuver, the distance required to perform the maneuver, and the distance the obstacle travels until the vehicle completes its maneuver. When Eq.~\ref{EqIneq2Proof} is re-arranged, it is Eq.~\ref{EqVcSensorIneq}.

The sigmoid curve timespan, $\tau_{f,s}$, satisfies the criteria for $\tau_{f,min}$ in Eq.~\ref{EqTauFTh} for either case in Eq.\ref{EqTdmax}. For the first condition in Eq.~\ref{EqTdmax}, both the previous and current maneuvers are the constant velocity case so $K_{n-1} = K_n = 0.5$ which simplifies condition 3 from Eq.~\ref{EqTauFTh} as
\begin{align}
	\frac{(1-K_n)\tau_{f,n}}{K_{n+1} \tau_{f,n+1}} &\leq 1 \\ \nonumber
	\frac{\tau_{f,n}}{\tau_{f,n+1}} &\leq 1
\end{align} 
\noindent where all the terms in the sigmoid curve timespan for $\tau_{f,n-1}$ and $\tau_{f,n}$ are the same except $\Delta \phi_{n-1}$ and $\Delta \phi_{n}$. Since we assume that $\Delta \phi_{n-1} \leq \Delta \phi_{n}$, the condition is satisfied.

For the second condition in Eq.~\ref{EqTdmax}, the sigmoid curve timespan is for a single maneuver which satisfies either condition 1 or 2 of Eq.~\ref{EqTauFTh}. Therefore, the sigmoid curve timespan is defined as
\begin{equation}
	\label{EqTauFsProof}
	\tau_{f,s} \geq \frac{2 c_3}{a_{max}}\frac{1}{2} \Delta \phi v_{c,s}
\end{equation}

Substituting the course change defined in Eq.~\ref{EqDelPhiObsKnown} into Eq.~\ref{EqTauFsProof} results in
\begin{equation}
	\label{EqIneq2TauFFull}
	\tau_{f,s} \geq \frac{c_3}{a_{max}} \left(\tan^{-1}\left(\frac{v_{o,max}}{\sqrt{v_{c,s}^2 - v_{o,max}^2}}\right) + \frac{\pi}{2}\right) v_{c,s}
\end{equation}
\noindent or for $\Delta \phi = \pi$, Eq.~\ref{EqTauFMinDef} results in
\begin{equation}
	\label{EqIneq2Tauf}
	\tau_{f,s} \geq \tau_{f,180}
\end{equation}
\noindent where
\begin{equation}
	\label{EqTauF180}
	\tau_{f,180} = \frac{ c_3}{a_{max}} \pi v_{c,s}
\end{equation}

Equations \ref{EqIneq2Proof} and \ref{EqIneq2TauFFull} (or \ref{EqIneq2Tauf}) are solved simultaneously for $v_{c,s}$ to bound the velocity based on the sensor range and update rate.
\end{proof}

\subsection{Obstacle Spacing Constraint, $v_{c,o}$}
\label{SubSecObsSpaceConstraint}
\begin{proof}
The third constraint on $v_c$ is due to the obstacle spacing, where the vehicle must clear obstacles by $r_c$ while following the course change definition from Sec.~\ref{SecCourseChange}. There must be enough distance between the obstacles, $r_{obs}$, for the vehicle to make a turn and clear the obstacle by $r_c$.

Figure \ref{FigObsSpacingStationaryAngle} shows an example scenario of a vehicle maneuvering around two obstacles. The minimum obstacle spacing, $r_{obs}$, must be $\geq 2 r_c + r_{180}$ to allow the vehicle enough space to respect a clearing of $r_c$ and also turn around between obstacles. 

\begin{figure}
	\centering
		\includegraphics[width=2.95in]{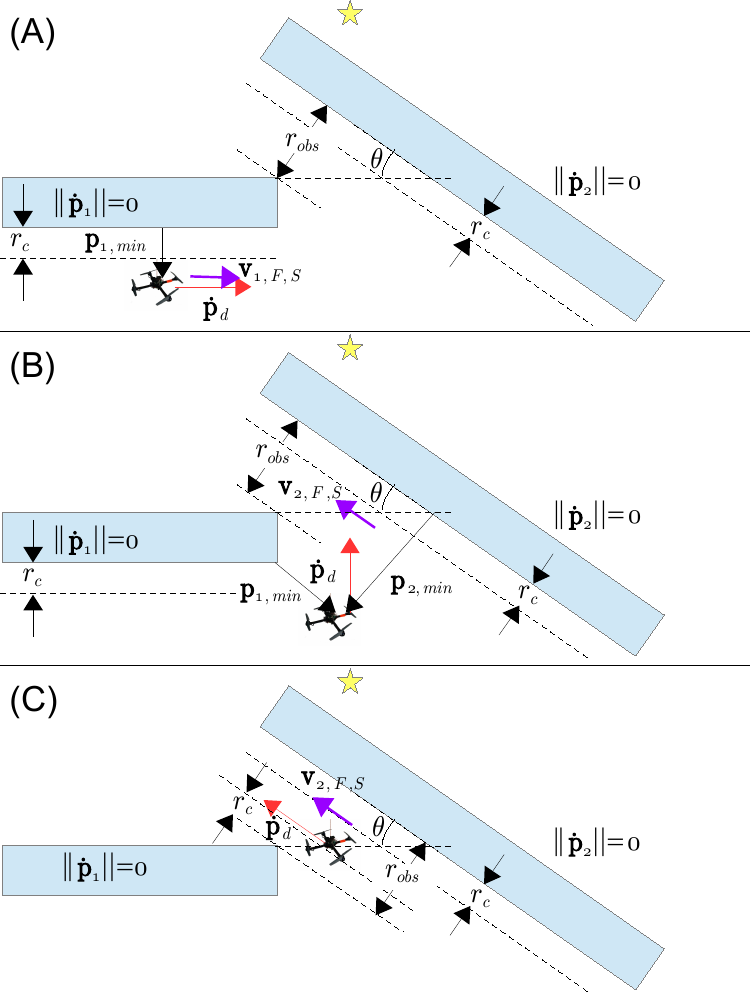}
		\caption{\label{FigObsSpacingStationaryAngle} Example scenario for a vehicle maneuvering around two stationary obstacles at some angle $\theta$ to reach the goal position. (A) The vehicle is $r_c$ away from the obstacle and matching $\mathbf{\dot{p}}_d = \mathbf{v}_{1,S}$. (B) The vehicle has cleared part of the obstacle and calculates a new desired $\mathbf{v}_{2,S}$; however, the vehicle cannot achieve this because it is constrained by $\mathbf{O}_{1}$. (C) The vehicle has now cleared the first obstacle to adjust to the desired $\mathbf{v}_{2,S}$ to traverse the second obstacle.}
	
\end{figure}

The obstacles in Fig.~\ref{FigObsSpacingStationaryAngle} are at some angle $\theta$ relative to one another. The bounding cases for the obstacle orientations are $\theta = 0$ and $\theta = \pi/2$. When $\theta = \pi/2$ the vehicle also makes a $\pi/2$ course change around obstacle 1. When $\theta = 0$, the vehicle may make a course change up to $\Delta \phi = \pi$ to fully traverse obstacle 1 from the initial course in Fig.~\ref{FigObsSpacingStationaryAngle}A. While the vehicle may eventually make the full $\Delta \phi = \pi$ maneuver, the tangent directions (Eqs.~\ref{EqPljS1} to \ref{EqPljS4}) that lead to the course change definition (Eq.~\ref{EqDelPhiPerTanDir}) only permit course changes up to $\Delta \phi = \pi/2$ as the vehicle maneuvers between obstacles. Therefore, both bounding cases lead to a course change of $\pi/2$. 

Equation \ref{EqDistanceSingleAxis} is used to determine the distance traveled for the $\pi/2$ maneuver, and the constant velocity case is still the constraining case (Appendix D). Based on Fig.~\ref{FigObsSpacingStationaryAngle} and the course change definition in Sec.~\ref{SecCourseChange}, the following inequality must be satisfied:
\begin{equation}
	\label{EqIneq3Proof}
	v_{c,o} \int_0^{\tau_{f,o}} \sin \phi(t) dt + v_{c,o} t_{d,max} + 2r_c \leq r_{obs}
\end{equation}
\noindent where $t_{d,max}$ is defined in Eq.~\ref{EqTdmax} to account for a delay in starting the obstacle maneuvering. When Eq.~\ref{EqIneq3Proof} is rearranged, it is Eq.~\ref{EqVcObsIneq}. Using Eq.~\ref{EqTauFN1ConstantVelGen} for $\Delta \phi = \pi/2$ results in
\begin{equation}
	\label{EqIneq3TauF}
	\tau_{f,o} = \frac{c_3}{a_{max}}\frac{1}{2} \pi v_{c,o}
\end{equation}
Equations \ref{EqIneq3Proof} and \ref{EqIneq3TauF} are solved simultaneously for $v_{c,o}$ to provide a bound that ensures the vehicle clears stationary obstacles by $r_c$ given a spacing of $r_{obs}$. This also holds for moving obstacles since the obstacle spacing for moving obstacles is $\geq r_{obs}$; therefore, the maximum course change is still $\pi/2$.
\end{proof}

\section*{Appendix C: Proof of Theorem 3}
\setcounter{subsection}{0}
Theorem 3 ensures that if the vehicle follows the established trajectory generation laws it reaches the goal position in finite time and clears all obstacles and vehicles by $r_c$. We examine each of these assertions separately.

\subsection{Goal Position Convergence}
\begin{proof}
From Assumption \ref{AssumpGoalPos} we know that the goal position is a valid target that is reachable without violating $r_c$ for other vehicles or obstacles. The vehicle only reaches the goal position once $||\mathbf{\dot{p}}_g|| = 0$; a condition we assume occurs in finite time during the flight. We define $e_{\phi_g} = \phi_g - \phi$ as the error in the course angle towards the goal position and $r_g$ as the distance to the goal position.

Maneuvering around obstacles and vehicles prevents the vehicle from always moving directly toward the goal position; therefore, $|\dot{e}_{\phi_g}| > 0$ at times. From Assumption \ref{AssumpLessCapable}, we know that $\min \limits_{k \in \mathcal{I}_{nr}} ||\mathbf{\dot{p}}_{k}|| > \max \limits_{ \in \mathcal{I}_{obs}}||\mathbf{\dot{p}}_{k}||$ so even the slowest vehicle is faster than the fastest obstacle. We can therefore assert that $v_{k,rem,S} > 0,~ \forall k \in \mathcal{I}_{obs}$ (Eq.~\ref{EqVrem}), so the vehicle is always capable of traversing the obstacle.

Additionally, from Assumption \ref{AssumpLessCapable} the obstacles are finite size, so the vehicle can traverse the obstacle in finite time until $\Delta \phi_g \in \mathbf{O}_{c,e}$ as long as the vehicle does not backtrack along the obstacle. The circumnavigation direction defined by Eq.~\ref{EqCircumDir} prevents the vehicle from backtracking along the obstacle since it defines a constant direction to traverse the obstacle, which restricts the course definition. The definition in Eq.~\ref{Eqzt} takes the cross product of the minimum distance vector, $\mathbf{p}_{k,min} - \mathbf{p}_d$, and the tangent direction vector, $\mathbf{p}_{k,S}$, for an obstacle $k$. If the cross product of the minimum distance vector and the vehicle velocity vector, $\mathbf{\dot{p}}_d$, is opposite the circumnavigation direction, then the vehicle traverses the obstacle in only one direction (the fixed circumnavigation direction) and thus does not backtrack.

To show that the cross product of $\mathbf{p}_{k,min} - \mathbf{p}_d$ and $\mathbf{\dot{p}}_d$ is opposite the circumnavigation direction, $\mathbf{z}_{k,S}$, the following equality must be satisfied:
\begin{equation}
	\label{EqPdCompareZlmin}
	\mathrm{sgn}\left(\left(\left(\mathbf{p}_{k,min} - \mathbf{p}_d\right) \times \mathbf{\dot{p}}_d \right) \cdot \mathbf{z}_I \right) \mathbf{z}_I = -\mathbf{z}_{k,S}
\end{equation}
\noindent or re-writing Eq.~\ref{EqPdCompareZlmin} in terms of Eq.~\ref{Eqzt} the inequality is
\begin{align}
	\label{EqPdcompareZlmin2}
	\mathrm{sgn}\left(\left(\left(\mathbf{p}_{k,min} - \mathbf{p}_d\right) \times \mathbf{\dot{p}}_d \right) \cdot \mathbf{z}_I \right) \mathbf{z}_I  = \\ \nonumber
	\mathrm{sgn}\left(\left(\left(\mathbf{p}_{k,min} - \mathbf{p}_d\right) \times \mathbf{p}_{k,S} \right) \cdot \mathbf{z}_I \right) \mathbf{z}_I 
\end{align}
The only difference in the left and right sides of Eq.~\ref{EqPdcompareZlmin2} is the $\mathbf{p}_{k,S}$ and $\mathbf{\dot{p}}_d$ terms. The vehicle velocity vector, $\mathbf{\dot{p}}_d$, is defined to match $\mathbf{v}_{k,S}$, which is determined from Eq.~\ref{EqVfDef} and can vary between $\mathbf{v}_{k,\perp,S}$ and $\mathbf{v}_{k,\parallel,S} = \mathbf{p}_{k,S}$, where $|\mathrm{angle}(\mathbf{p}_{k,min}-\mathbf{p}_d,\mathbf{v}_{k,\perp,S})| > |\mathrm{angle}(\mathbf{p}_{k,min}-\mathbf{p}_d,\mathbf{v}_{k,\parallel,S})|$ for convex obstacles. However, since we know that $||\mathbf{\dot{p}}_d|| > ||\mathbf{\dot{p}}_{k}||,~\forall k \in\mathcal{ ID}_{obs}$, then even if $\mathbf{v}_{k,\perp,S} = \mathbf{\dot{p}}_{k}$, $\mathbf{\dot{p}}_d$ is still in a direction between $\mathbf{v}_{k,\perp,S}$ and $\mathbf{p}_{k,S}$. Furthermore, the maximum angle between $\mathbf{v}_{k,\perp,S}$ and $\mathbf{p}_{k,min}-\mathbf{p}_{d}$ is less than $\pi$ for convex obstacles. Therefore, Eq.~\ref{EqPdcompareZlmin2} is always satisfied when $\mathbf{O}_{co} = \emptyset$ and/or $\mathbf{O}_{cv} = \emptyset$ for convex obstacles.


In the case where there are non-convex obstacles and/or critical obstacles and vehicles present, then the vehicle may temporarily violate Eq.~\ref{EqPdCompareZlmin}. This could occur for a case where two vehicles that are both $\approx r_c$ from an obstacle approach each other from opposite directions. The vehicle velocities should be equal since there are both obstacles and vehicles within sensor range; therefore, the maneuvering vehicle likely makes a course change $> \pi/2$ to navigate out of the way. This could also occur for a non-convex obstacle with a velocity vector opposite the vehicle velocity when the vehicle chooses the conservative tangent direction solution. In either of these cases, Eq.~\ref{EqPdCompareZlmin} is violated, but only temporarily. To ensure that the vehicle reestablishes the proper circumnavigation, the following must be satisfied:
\begin{equation}
	\mathrm{sgn}\left(\mathrm{angle}(\mathbf{\dot{p}}_d,\mathbf{p}_{k,min})\right) \neq \mathrm{sgn}\left(\mathrm{angle}(\mathbf{\dot{p}}_d,\mathbf{v}_{k,S})\right)
\end{equation}
\noindent where it is assumed that there is an appreciable angle between $\mathbf{\dot{p}}_d$ and $\mathbf{p}_{k,min}$. If this condition is not satisfied, then the course change is modified to preserve $r_c$ spacing and circumnavigation direction as follows
\begin{equation}
	\Delta \phi^* = -\mathrm{sgn}({z}_{k,S}) |2 \pi - \Delta \phi|
\end{equation}
This upholds the circumnavigation direction to continue to traverse the obstacle in one direction.

Additionally, as the vehicle traverses obstacles, parts of the obstacle may go out of sensor range. For non-convex obstacles this could be detrimental because the vehicle may compute the goal position as valid when there is actually part of the obstacle blocking the path. To ensure that the vehicle continues to traverse an obstacle, the stored extent point discussed in Sec.~\ref{SecCourseChangeOneObs} and the modified feasible set $\mathbf{O}_{k,S}$ in Eq.~\ref{EqSetOrtRefined} ensure that even if the point is out of sensor range it is still constraining the course changes so the vehicle continues to traverse the obstacle in one direction and clear it. If the obstacle is moving, the stored extent point is projected forward in time by the estimated obstacle velocity.

If it can clear one obstacle in finite time and there are a finite number of obstacles (Assumptions \ref{AssumpLessCapable}), then the vehicle is able to clear all obstacles between its starting point and the goal position in finite time. 

Even though the other vehicles do not have a fixed circumnavigation direction, the vehicle still clears other vehicles in finite time. It is assumed that the other vehicles are also headed to goal positions, thus allowing the maneuvering vehicle to clear it. Even if a vehicle has slowed down to match velocity, eventually the vehicle either reaches its goal, or the non-maneuvering vehicle reaches its goal and the maneuvering vehicle can pass.

Once the vehicle has cleared all obstacles, Eq.~\ref{EqFinalDelPhi} results in $\Delta \phi = \Delta \phi_g$, whereby $e_{\phi_g} \rightarrow 0$. Since the goal position is reachable in finite time, the vehicle reaches $\mathbf{p}_g$ in finite time, at which point the algorithm generates a trajectory to bring the vehicle to the goal position in finite time. Thus, $r_g = 0$ for $||\mathbf{\dot{p}}_g || = 0$.
\end{proof}

\subsection{Clearance Radius Guarantee}
\begin{proof}
The clearance radius guarantee is dependent on both the velocity bound as well as the projected points that use the $r_c$ (for obstacles) or $r_{c,k}^*$ (for vehicles) clearance circle. The velocity bound addresses the vehicle's approach to obstacles to ensure that it has sufficient distance to make turns, as well as ensuring that for a minimum obstacle spacing the vehicle can maneuver around obstacles safely. 

Additionally, because of the unpredictable course change of other vehicles, the $r_{c,k}^*$ distance is calculated similar to the sensor constraint in Theorem 2. We assume that the vehicle is just outside of the $r_{c,k}^*$ circle so it travels at its current velocity until the next sensor update $\Delta T_s$ later. If re-prioritization is necessary, as discussed in Sec.~\ref{SecCourseChange}, then there is an additional $\Delta T_s$ before the vehicle re-prioritizes and generates a new trajectory. Assuming a worst case where the other vehicle is going the same speed and has changed course to come directly towards the vehicle, the vehicle must make a 180\degree ~turn and still clear the other vehicle by $r_c$. The $r_c$ values may be different for the two vehicles, and since this information is shared, the maneuvering vehicle takes the maximum clearance. Equation \ref{EqTheRcStar} provides the $r_{c,k}^*$ definition for another vehicle $k \in \mathcal{I}_{nr}$ as:
\begin{equation}
	\label{EqTheRcStar}
	r_{c,k}^*=\mathrm{max}\left(r_{c,k},r_c\right) + r_{180} + ||\mathbf{\dot{p}}_k|| (2\Delta T_s + \tau_{f,180} + \Delta T_c)
\end{equation}
\noindent where $r_c$ is the clearance radius of the current vehicle, $r_{c,k}$ is the clearance radius of sensed vehicle $k$, $r_{180}$ is defined by Eq.~\ref{EqDistanceSingleAxis} and $\tau_{f,180}$ is defined in Eq.~\ref{EqTauF180}. 

For both obstacles and vehicles the extension by $r_c$ and $r_{c,k}^*$, respectively, shown in Fig.~\ref{FigPadObsByRC} and defined in Sec.~\ref{SecCourseChangeOneObs}, ensures that the candidate projected extent points do not violate $r_c$. The selection of the conservative or aggressive tangent direction from Eq.~\ref{EqAggresConserv} then ensures that the vehicle clears the obstacle by $r_c$. 


\end{proof}

\section*{Appendix D: Proof of Acceleration Curve Unimodality}
The proof that the acceleration curve has a single peak guarantees not only that there is a solution for $H$ (from Eq.~\ref{EqHCubic}) that is physically meaningful, but that it is unique.  

\begin{proof}
We can show that the acceleration curve has a single maximum for all velocity ($0 \leq |\Delta v| \leq v_c$) and course ($0 \leq |\Delta \phi| \leq \pi$) changes. The solution for $H$ in Eq.~\ref{EqHCubic} therefore only has one real solution that satisfies $|H| \leq 1-\varepsilon$. Since the terms in the acceleration curve (Eq.~\ref{EqAccelMaxMag}) are squared, we can perform the analysis for positive $\Delta v$ and $\Delta \phi$ without loss of generality. Likewise, we perform the analysis for the terms within the square root (which are always positive) so the unimodality is preserved once the square root is taken. 

We start by examining the two terms within the square root of Eq.~\ref{EqAccelMaxMag} separately. Let us examine the second term first, which is $\dot{v}^2$. This term is maximized where $\frac{d}{dt} \dot{v}^2 = 2 \dot{v} \ddot{v} = 0$. Expanding this equation gives
\begin{equation}
	(d_1 d_2 (1-H^2)) (-2 d_1 d_2 H (1-H^2)) = 0
\end{equation}
\noindent where the solutions are $H = 0, \pm 1$. The physically meaningful solution is $H = 0$, corresponding to $K = 0.5$, which is expected since the expression for $\dot{v}$ is symmetric about $K = 0.5$. We write these solutions as $H_1 = 0$ and $K_1 = 0.5$ for later reference in the analysis.

Next, we examine the first term, $v^2 \dot{\phi}^2$, which is maximized where $\frac{d}{dt} v^2 \dot{\phi}^2 = 0$. Expanding this equation results in
\begin{equation}
	\label{EqDerivAccelMaxMagTerm1}
	2v \dot{\phi} \left(v \ddot{\phi} + \dot{v} \dot{\phi}\right) = 0
\end{equation}
\noindent After substituting the definitions for $K$ and $H$ from Eqs.~\ref{EqKDef} and \ref{EqHDef} and the sigmoid functions from Eqs.~\ref{EqSigPhi} and \ref{EqSigV} into Eq.~\ref{EqDerivAccelMaxMagTerm1}, we examine the two terms of Eq.~\ref{EqDerivAccelMaxMagTerm1} separately. The first term gives
\begin{align}
	2v\dot{\phi} &= 0 \nonumber \\
	(d_1H + d_4)\left(c_1 c_2 \left(1-H^2\right)\right) &= 0 
\end{align} 
\noindent where $H = \pm 1$ for all three roots since $d_1 = d_4$. These roots do not satisfy $|H| \leq 1-\varepsilon$, so they are not physically meaningful solutions.

The second term of Eq.~\ref{EqDerivAccelMaxMagTerm1} is expanded as follows:
\begin{align}
	\left(v \ddot{\phi} + \dot{v} \dot{\phi}\right) = 0 \nonumber \\
	(d_1 H + d_4)\left(-2 c_1 c_2 H \left(1-H^2\right)\right) + \nonumber \\
	\left(d_1 d_2 \left(1-H^2 \right) \right) \left(c_1 c_2 \left(1-H^2 \right)\right) = 0 \nonumber \\
	c_1 c_2 \left(1-H^2\right)\left(-2H\left(d_1 H + d_4\right) + d_1 \left(1-H^2\right)\right) = 0
\end{align}
\noindent where there are again roots at $H = \pm 1$. There is also a quadratic in $H$ which we isolate and re-write as
\begin{equation}
	-3 d_1 H^2 - 2 d_4 H + d_1 = 0
\end{equation}
\noindent where the roots are
\begin{equation}
	\label{EqHRootsAccelCurveShape}
	H = \frac{-2d_4 \pm \sqrt{4 d_4^2 + 12 d_1^2}}{-6 d_1}
\end{equation}
\noindent Since $d_4 = d_1$, Eq.~\ref{EqHRootsAccelCurveShape} simplifies to $H = -1$ or $H = \frac{1}{3}$. The physically meaningful solution is $H = \frac{1}{3}$ corresponding to $K = 0.5456$. We write these solutions as $H_2 = \frac{1}{3}$ and $K_2 = 0.5456$ for reference later in the analysis.

Now that we know the locations of the maxima for each term and we also know the terms have the same timespan, $\tau_f$, we compute the intersection point of the curves to determine how the curves add together. 

The intersection point is where $v^2 \dot{\phi}^2 = \dot{v}^2$. We again substitute $K$ and $H$ into the sigmoid functions, and simplify to produce Eq.~\ref{EqHIntersectAccelCurveShape}:
\begin{align}
	v^2 \dot{\phi}^2 &= \dot{v}^2 \nonumber \\
	(d_1H_{int} + d_4)^2\left(c_1c_2(1-H_{int}^2)\right) &= \left(d_1 d_2 (1-H_{int}^2)\right)^2 \nonumber \\
	(d_1H_{int} + d_4)^2 c_1^2 &= d_1^2 \nonumber \\
	d_1(H_{int}+1) c_1 &= d_1 \nonumber \\
	H_{int}+1 &= \frac{1}{c_1} \nonumber \\
	\label{EqHIntersectAccelCurveShape}
	H_{int} &= \frac{1}{c_1} - 1 = \frac{2} {\Delta \phi} - 1
\end{align}

We know the maximum acceleration locations are at $H_1 = 0$ and $H_2 = \frac{1}{3}$. From these solutions we can determine the three regions that bound the course changes and corresponding curve intersection points as shown in Fig.~\ref{FigConstantAccelShape} and given by
\begin{align}
	&\mathrm{region}~\mathrm{1}~~H_{int} < 0 \\
	\label{EqHVarySectionB}
	&\mathrm{region}~\mathrm{2}~~0\leq H_{int} \leq \frac{1}{3} \\
	&\mathrm{region}~\mathrm{3}~~ H_{int} > \frac{1}{3}
\end{align}

\begin{figure}
	\centering
		\includegraphics[width=2.95in]{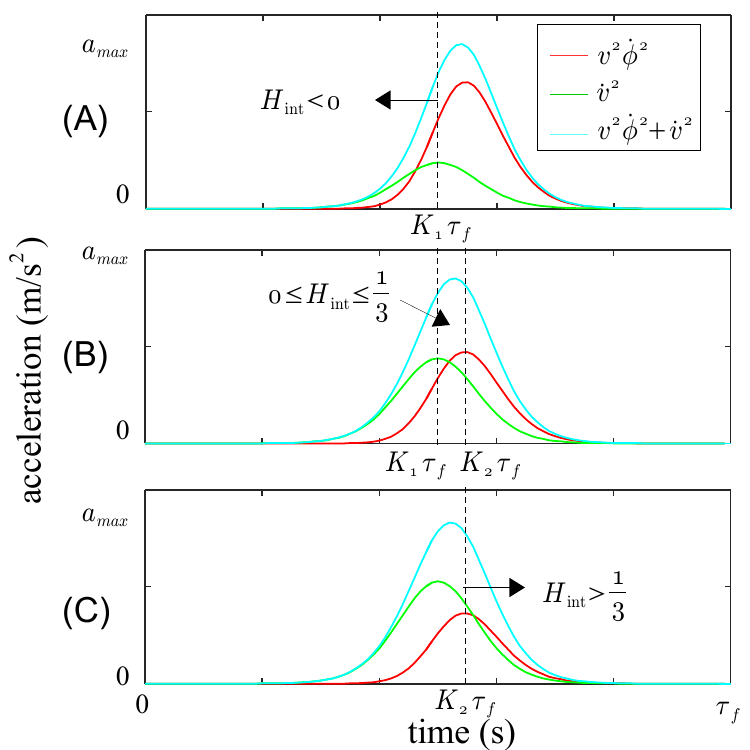}
		\caption{\label{FigConstantAccelShape} Example curves to illustrate the locations of the intersection points between the two curves. (A) The intersection point is prior to the peak of the $\dot{v}^2$ curve which occurs for any $H_{int} < 0$. (B) The intersection point is between the peaks of the curves which occurs for $0 \leq H \leq \frac{1}{3}$. (C) The intersection point is after the peak of the $v^2 \dot{\phi}^2$ curve which occurs for $H_{int} > \frac{1}{3}$. }
	
\end{figure}

In each of the regions, prior to the first peak at $K_1 \tau_f$, both curves are increasing, and after the second peak at $K_2 \tau_f$, both curves are decreasing. Therefore, any maxima of the summed curve must occur between $K_1 \tau_f$ and $K_2 \tau_f$. 

For the cases illustrated in Fig.~\ref{FigConstantAccelShape}A and C, the intersection of the curves is outside of the region between $K_1 \tau_f$ and $K_2 \tau_f$, where the slopes of both curves are either positive or negative, respectively. For these cases, the sum of the curves produces a single maxima since one curve is monotonic in these intervals.

In the case of the region in Fig.~\ref{FigConstantAccelShape}B, the intersection point is between $K_1 \tau_f$ and $K_2 \tau_f$. To have multiple minima, the following two inequalities must be satisfied
\begin{align}
	\label{EqSinglePeakIneq1}
	v(K_{int}\tau_{f})^2 &\dot{\phi}(K_{int}\tau_{f})^2 + \dot{v}(K_{int}\tau_{f})^2\leq v(K_1\tau_{f})^2 \dot{\phi}(K_1\tau_{f})^2 + \dot{v}(K_1\tau_f)^2 \\
	\label{EqSinglePeakIneq2}
	v(K_{int}\tau_{f})^2 &\dot{\phi}(K_{int}\tau_{f})^2 + \dot{v}(K_{int}\tau_{f})^2 \leq v(K_2\tau_{f})^2 \dot{\phi}(K_2\tau_{f})^2 + \dot{v}(K_2\tau_{f})^2
\end{align}
These inequalities are further simplified by substituting Eq.~\ref{EqHIntersectAccelCurveShape}, and the sigmoid functions into Eqs.~\ref{EqSinglePeakIneq1} and \ref{EqSinglePeakIneq2}
\begin{align}
	\label{EqSinglePeakRefinedIneq1}
	2 \left(H_{int}^2-1\right)^2 \leq \left(H_1^2 -1\right)^2 + \frac{\left(H_1+1\right)^2}{\left(H_{int} + 1\right)^2} \left(H_1^2 - 1\right)^2 \\
	\label{EqSinglePeakRefinedIneq2}
	2 \left(H_{int}^2-1\right)^2 \leq \left(H_2^2 -1\right)^2 + \frac{\left(H_2+1\right)^2}{\left(H_{int} + 1\right)^2} \left(H_2^2 - 1\right)^2 
\end{align}

The values for $H_{int}$ vary according to Eq.~\ref{EqHVarySectionB} and we know the values for $H_1$ and $H_2$. Over this range the inequalities are not satisfied, meaning that there are not multiple maxima. We therefore conclude that the acceleration curve has a single maximum and a unique solution for $H$. 

This result is used in Appendices A and B to prove that the sigmoid curve timespan does not violate the maximum thrust, and that the cruise velocity is set appropriate for the vehicle's hardware constraints and environment.
\end{proof}

\section*{Appendix E: Proof of Constant Velocity as the Most Constraining Case}
The analysis to prove that the constant velocity case is the most constraining case for course changes $\Delta \phi \geq \pi/2$ is important because it guarantees that Appendix B uses the maximum distance traveled for the cruise velocity sensor constraint.

Consider Fig.~\ref{FigSigAreaExample}, which shows that the area under the curve (distance traveled) for the constant velocity case is always greater than the changing velocity case over the same or longer $\tau_{f,s}$.

\begin{figure}
	\centering
		\includegraphics[width=2.95in]{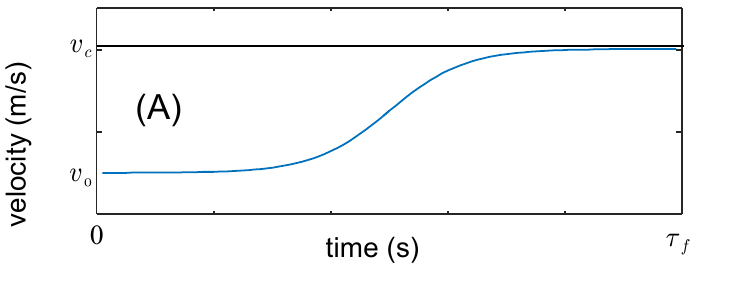}
		\caption{\label{FigSigAreaExample} Example sigmoid function for velocity change showing difference in area under the curve for constant ($\Delta v = 0$) and changing ($|\Delta v| > 0$) velocity cases over the same timespan.}
	
\end{figure}

To show that the constant velocity case has a longer timespan than the changing velocity case, we consider the solution for $\tau_{f,min}$ from Eq.~\ref{EqTauFMinDef}, where the only difference between the two cases is the $S_{traj}$ term defined in Eq.~\ref{EqStrajDef}. 

We know that there are two cases for velocity change, either acceleration or deceleration. We also assume that any change is to either accelerate to $v_c$ or decelerate from $v_c$. This allows us to generically define the following terms:
\begin{align}
	v_n &= v_c - \Delta v,~~~\mathrm{acceleration}\\
	v_n &= v_c,~~~\mathrm{deceleration} \\
	\label{Eqd1vc}
	d_1 &= \pm \frac{1}{2} \Delta v \\
	\label{Eqd4vc}
	d_4 &= v_c - \frac{1}{2} \Delta v
\end{align}
Substituting Eqs.~\ref{Eqd1vc} and \ref{Eqd4vc} into Eq.~\ref{EqStrajDef} results in
\begin{equation}
	\label{EqStrajvc}
	\left(c_1\left(v_c - \frac{1}{2}\Delta v \pm \frac{1}{2} \Delta v H\right)\right)^2 + \left(\frac{1}{2} \Delta v \left(1-H^2\right)\right)^2
\end{equation}
\normalfont
\noindent Further manipulation of Eq.~\ref{EqStrajvc} results in
\begin{equation}
	c_1^2 v_c^2 \left(\left(1-H^2\right)^2 \left(1\pm \frac{\Delta v}{v_c} (H\mp 1) + 
	\frac{1}{4} \frac{\Delta v^2}{v_c^2} (H \mp 1)^2 + \frac{1}{4} \frac{\Delta v^2}{v_c^2} \frac{1}{c_1^2}\right)\right)
\end{equation}
For comparison, the $S_{traj}$ terms for $\Delta v = 0$ is defined in Eq.~\ref{EqTauFN1ConstantVelGen}; therefore, for the $S_{traj}$ term to be less for the $|\Delta v|>0$ case, the following inequality must be satisfied:
\begin{equation}
	\label{EqStrajConservIneq}
	\left((1-H^2)^2\left(1 \pm \frac{\Delta v}{v_c}(H\mp 1) +  \frac{1}{4} \frac{\Delta v^2}{v_c^2} (H\mp1)^2 + \frac{1}{4} \frac{\Delta v^2}{v_c^2} \frac{1}{c_1^2}\right)\right) \leq 1 \\
\end{equation}
\noindent where the solution to $H$ from Eq.~\ref{EqHCubic} is dependent on $\Delta v/v_c$ and $c_1= 0.5\Delta \phi$. For any specific case these are all known quantities; however, the analytical expression for the generic case is unwieldy and does not provide useful insight. Instead, we know the bounds for $H$ and $\Delta v/v_c$ as:
\begin{align}
	\label{EqHmaxLimits}
	-H_{max} &\leq H \leq H_{max} \\
	\label{EqDelVvcLims}
	0 &\leq \frac{\Delta v}{v_c} \leq 1 
\end{align}
\noindent where $\pm H_{max}$ are the roots of
\begin{equation}
	\label{EqHCubicSimplified}
	-3c_1^2 H^3 \mp 5c_1^2 H^2 - (c_1^2 + 2)H \pm c_1^2 = 0
\end{equation}
\noindent which is derived from Eq.~\ref{EqHCubic} in the case of $\Delta v = \pm v_c$. 

To develop the bounds on cruise velocity based on sensor range, we know that the course change must satisfy $\pi/2 \leq \Delta \phi \leq \pi$. The solution to Eq.~\ref{EqHCubicSimplified} for these inputs gives $0.19\leq H_{max}\leq 0.29$. Even though $H$ and $\Delta v/v_c$ are not independent, we can conservatively treat them independently and evaluate Eq.~\ref{EqStrajConservIneq} for $0.19 \leq H_{max} \leq 0.29$ and the expected range of $\Delta v/v_c$ from Eq.~\ref{EqDelVvcLims}. For all values in this range the inequality is satisfied, so the constant velocity case is the most constraining case for vehicle distance traveled.

This result is used in Appendix B to prove the cruise velocity constraint for the vehicle's sensor parameters is set appropriately for the environment.

\bibliographystyle{unsrt}
\bibliography{C:/Users/cole/Documents/ColeThesisWork/Documentation/ReferencesBibFiles/Controls}

\vspace{-4in}
\begin{IEEEbiography}[{\includegraphics[width=1in]{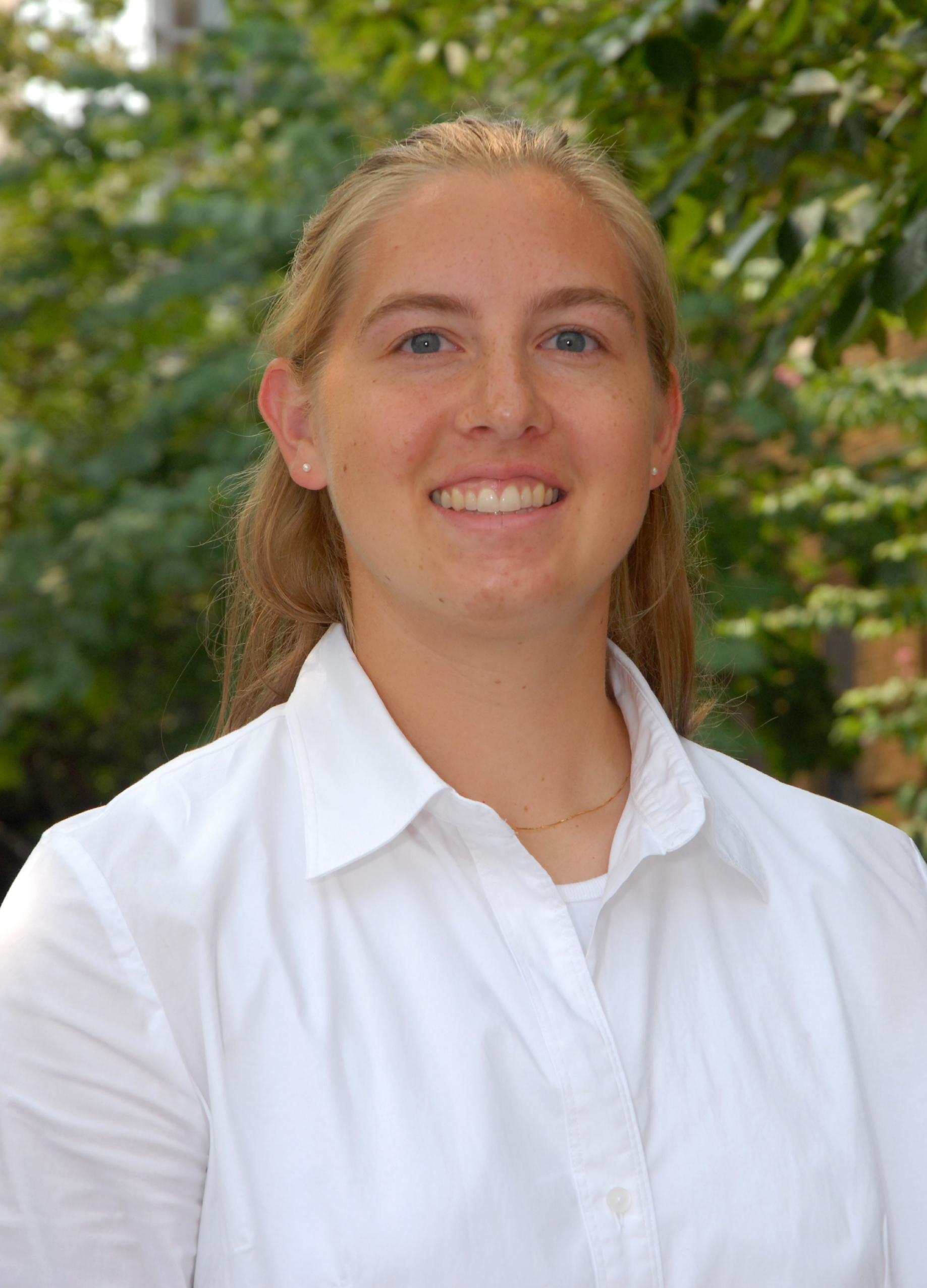}}]{Kenan Cole}
Kenan Cole received her B.S. and M.S. degrees in mechanical engineering from The George Washington University (GUW) in 2007 and 2011, respectively. She is currently pursuing a Ph.D. at GWU in mechanical engineering focused on multi-vehicle controls. Her research interests include vehicle controls, formation controls, and modeling environmental disturbances. 
\end{IEEEbiography}
\vspace{-4in}
\begin{IEEEbiography}[{\includegraphics[width=1in]{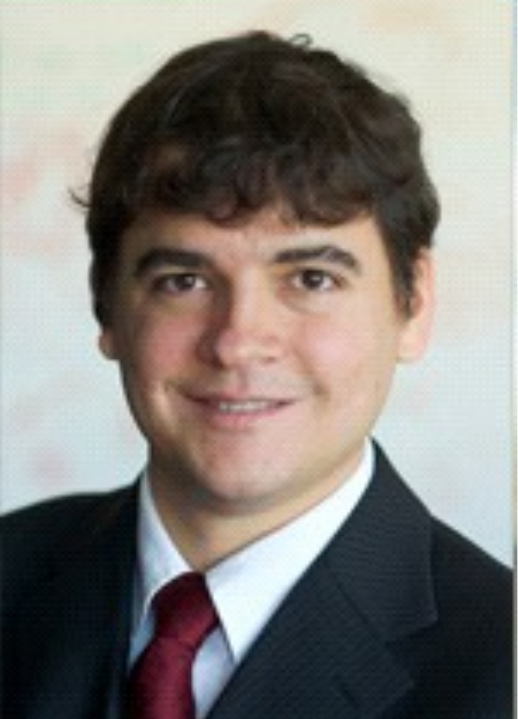}}]{Adam Wickenheiser}
Adam M. Wickenheiser (M’08) received the B.S. degree in mechanical engineering (with a minor in applied mathematics) in 2002 and the M.S. and Ph.D. degrees in aerospace engineering from Cornell University, Ithaca, in 2006 and 2008, respectively.

Since 2010, he has been the Faculty of the Department of Mechanical \& Aerospace Engineering at the George Washington University, Washington, DC, where he is currently an Assistant Professor.  From 2008-2009, he was a postdoctoral associate with the Sibley School of Mechanical \& Aerospace Engineering at Cornell University.  His current research interests include bio-inspired flight, multi-functional materials and systems, and energy harvesting for autonomous systems.

Prof. Wickenheiser has served as the Chair of the Energy Harvesting Technical Committee of the American Society of Mechanical Engineers (ASME) from 2014-2016, and is presently a member of the International Organizing Committee for the International Conference on Adaptive Structures and Technologies (ICAST). He was the recipient of the 2011 Intelligence Community Young Investigator Award.
\end{IEEEbiography}

\end{document}